\documentclass[11pt,reqno]{amsart}
\usepackage[letterpaper, margin=1in]{geometry}
\usepackage{amsmath,amscd,amssymb}
\usepackage{url}
\usepackage{graphicx}

\usepackage{color}
\usepackage{tikz-cd}
\usepackage{enumerate}
\usepackage[pagebackref,colorlinks,citecolor=blue,linkcolor=magenta]{hyperref}
\usepackage[linesnumbered,ruled]{algorithm2e}
\RequirePackage{amsthm,amsmath,amsfonts,amssymb}
\usepackage[utf8]{inputenc}
\usepackage{chngcntr}
\usepackage{caption}
\usepackage{subcaption}
\usepackage{placeins}
\usepackage{xcolor}
\usepackage{tikz}
\usetikzlibrary{decorations.pathreplacing,calligraphy}
\usetikzlibrary{arrows}
\usepackage{nicematrix}
\usepackage{bbm}

\newtheorem{theorem}{Theorem}
\newtheorem*{theorem*}{Theorem}
\numberwithin{theorem}{section}
\newtheorem{proposition}[theorem]{Proposition}

\newtheorem{corollary}[theorem]{Corollary}
\newtheorem{lemma}[theorem]{Lemma}

\newtheorem{question}[theorem]{Question}
\theoremstyle{definition}
\newtheorem{definition}[theorem]{Definition}
\newtheorem*{definition*}{Definition}

\newtheorem{example}[theorem]{Example}
 \newtheorem{problem}[theorem]{Problem}

\newcommand{\zz}{\mathbb{Z}}

\newcommand{\rr}{\mathbb{R}}

\newcommand{\kk}{\mathbb{K}}

\newcommand{\ca}{\mathcal{A}}

\newcommand{\cm}{\mathcal{M}}

\newcommand{\ct}{\mathcal{T}}
\newcommand{\cg}{\mathcal{G}}

\newcommand{\ch}{\mathcal{H}}

\newcommand{\patp}{\mathcal{P}}

\DeclareMathOperator{\CIM}{CIM}
\DeclareMathOperator{\cim}{CIM}
\DeclareMathOperator{\init}{in}
\DeclareMathOperator{\lift}{Lift}

\DeclareMathOperator{\pa}{pa}
\DeclareMathOperator{\conv}{conv}

\DeclareMathOperator{\meq}{Pat}
\DeclareMathOperator{\parting}{part}
\begin{document}

\title[Quasi-Independence Gluing]{Toric Ideals of Characteristic Imsets via Quasi-Independence Gluing}

\date{\today}

\author{Benjamin Hollering}
\address{Max Planck Institute for Mathematics in the Sciences \\
Inselstra{\ss}e 22, 04103 Leipzig, Germany}
\email{benjamin.hollering@mis.mpg.de}

\author{Joseph Johnson}
\address{North Carolina State University, Raleigh, NC 27695}
\email{jwjohns5@ncsu.edu}

\author{Irem Portakal}
\address{Technische Universit\"at M\"unchen, 85748 Garching b. München, Boltzmannstr. 3.,  Germany}
\email{mail@irem-portakal.de}

\author{Liam Solus}
\address{Institutionen f\"or Matematik, KTH, SE-100 44 Stockholm, Sweden}
\email{solus@kth.se}

\begin{abstract}
Characteristic imsets are 0-1 vectors which correspond to Markov equivalence classes of directed acyclic graphs. The study of their convex hull, named the characteristic imset polytope, has led to new and interesting geometric perspectives on the important problem of causal discovery. In this paper we begin the study of the associated toric ideal.  We develop a new generalization of the toric fiber product, which we call a quasi-independence gluing, and show that under certain combinatorial homogeneity conditions, one can iteratively compute a Gr\"obner basis via lifting. For faces of the characteristic imset polytope associated to trees, we apply this technique to compute a Gr\"obner basis for the associated toric ideal. We end with a study of the characteristic ideal of the cycle and propose directions for future work.

%\textcolor{red}{L: Goals: (1) Find a quadratic and square-free Grobner basis for $\CIM_G$ for $G$ a tree. (2) What is the degree of this variety? (3) What do the $h^\ast$-polynomials look like? ... at least some examples other than the path graph?}
\end{abstract}

\keywords{characteristic imset, quasi-independence, toric fiber product, polytope}
\subjclass[2020]{62E10, 62H22, 62D20, 62R01, 13P25, 13P10}

\maketitle
\thispagestyle{empty}

%---SECTION: Introduction---
\section{Introduction}
\label{sec: intro}
% Directed ayclic graphs are widely used throughout the sciences to model conditional independence and causal relationships between random variables {\color{purple} TOCITE}. 
Given a directed acyclic graph (DAG) $\cg = ([n], E)$ with vertices $[n] := \{1,\ldots, n\}$ and edges $E$ and a collection of jointly distributed random variables $(X_1,\ldots, X_n)$ with probability density function $f(x_1,\ldots, x_n)$, we say that $f(x_1,\ldots, x_n)$ is \emph{Markov} to the DAG $\cg$ if
\[
f(x_1,\ldots, x_n) = \prod_{i=1}^nf(x_i | x_{\pa_\cg(i)}),
\]
where $\pa_\cg(i) = \{j \in [n] : j\rightarrow i\in E\}$ is the set of \emph{parents} of $i$ in $\cg$.  

Directed acyclic graphical models play a fundamental role in modern data science and artificial intelligence through their applications in probabilistic inference \cite{koller2009probabilistic} and causality \cite{pearl2009causality}. In such fields, the combinatorics of the graph can be used to determine complexity bounds for probabilistic inference algorithms \cite[Chapter 9]{koller2009probabilistic} and estimators for the causal effect of one variable in the system on another \cite[Chapter 3]{pearl2009causality}. Given these applications, it is valuable to have a DAG representation of a data-generating distribution, especially if such a representation can be learned directly from data. Unfortunately, many different DAGs may encode the same set of conditional independence (CI) statements \cite{andersson1997markov, verma1990equivalence, verma1992algorithm}. Two such DAGs which encode the same set of CI statements are called \emph{Markov equivalent} and so the main goal is to recover the Markov equivalence class of the DAG which generated the data. 
This is the basic problem of \emph{causal discovery}.  
A wide variety of methods for doing so have been proposed and are primarily based on the combinatorics of DAGs \cite{chickering2002optimal, kuipers2022efficient, LRS2022a, solus2017consistency, spirtes1991algorithm, tsamardinos2006max}. 

More recently, efforts were made to rephrase the problem of causal discovery as a linear program by embedding DAGs as 0/1-vectors in real-Euclidean space called \emph{characteristic imsets} \cite{studeny2010characteristic}. The characteristic imset vector of a DAG uniquely encodes its Markov equivalence class.
The convex hull of all such vectors for DAGs with $n$ vertices is called the \emph{characteristic imset polytope} $\cim_n$. 
Although the linear program approach that motivated the definition of $\cim_n$ remains infeasible for even reasonably small numbers of variables, due to the fact that there is no complete characterization of the facets of $\cim_n$, the geometry of $\cim_n$ appears to play a fundamental role in the general problem of causal discovery. 
For instance, it was recently show that the moves used by popular search-based optimization causal discovery algorithms all correspond to edges of $\cim_n$; hence, such methods amount to walks along the edge graph of $\cim_n$ \cite{LRS2022a}.
In \cite{LRS2022a}, it was also shown that the polytope $\cim_G$, defined as the convex hull of all characteristic imsets $c_\cg$ where $\cg$ has adjacencies corresponding to the edges in the undirected graph $G$, is a face of $\cim_n$. 
% In \cite{LRS2022a}, it was also shown that the 
% \[
% \cim_G = \conv(c_\cg : \cg \mbox{ has skeleton $G$}) 
% \]
% is a face of $\cim_n$ for any undirected graph $G = ([n], E)$.  
% Here, the \emph{skeleton} of a DAG $\cg = ([n], E)$ is the undirected graph $G = ([n], E)$ whose edges indicate that either $i\rightarrow j$ or $i \leftarrow j$ is an edge of $\cg$. 
In \cite{LRS2022b}, the edges of $\cim_G$ for $G$ a tree are completely characterized, and it is shown that algorithms using these edges to search for an optimal Markov equivalence class with skeleton $G$ outperform classic causal discovery methods for learning directed trees. 
% {\color{purple} Insert more info on how many causal discovery algorithms are geometric and the importance of the polyhedral structure to causality}
% {\color{joegreen} We haven't said what $\cim_G$ is yet}

While significant work has already been done on the polyhedral structure of $\cim_G$, there has been little attention paid to its toric geometry. There is a strong and well-established link between the algebraic properties of toric ideals and the combinatorial properties of their associated polytopes \cite{sturmfels1996Grobner}. In this paper, we begin the study of the toric ideal $I_G$ which is naturally associated to $\cim_G$ with the hope that it can be used to better understand $\cim_G$. We begin by introducing a new generalization of the \emph{toric fiber product} of \cite{sullivant2007toric} which is defined as follows. 
\begin{definition}[Definition~\ref{def: QIG}]
Let $Q \subset [r] \times [s]$ and $I \subset \kk[x_j ~|~ j \in [r]]$ and $J \subset \kk[y_k ~|~ k \in [s]]$ be homogeneous ideals. The \emph{quasi-independence gluing} of $I$ and $J$ with respect to $Q$ is 
\[
I \times_Q J := \phi_Q^{-1}(I+J)
\]
where $\phi_Q$ is the map
\begin{align*}
\phi_Q: \kk[z_{jk} ~|~ (j, k) \in Q]& \to \kk[x_j, y_k ~|~ j \in [r] , k \in [s]]\\
 z_{jk}& \mapsto x_j y_k.
\end{align*}
\end{definition}
Similar to the toric fiber product, we show that the Gr\"obner basis of $I \times_Q J$ can be computed by lifting polynomials from the Gr\"obner bases of ideals $I$ and $J$ and adding in extra binomials which come from the monomial map $\phi_Q$. The Gr\"obner bases of the ideals $I_Q = \ker(\phi_Q)$ are completely known and have been studied in algebraic statistics \cite{aoki2012markov, coons2021quasi} and commutative algebra as the \emph{toric ideals of edge rings} of bipartite graphs \cite{de1995grobner,Herzog2018,por21,villarreal}. Unlike the toric fiber product though, we allow for $\phi_Q$ to be any monomial homomorphism which arises as the parameterization of a \emph{quasi-independence model} and do not require that $I$ and $J$ are homogeneous with respect to multigradings. 
We instead require a weaker condition which we call $Q$-homogeneity that ensures that polynomials in $I$ and $J$ can be lifted to $I \times_Q J$. 

We then show that if $T$ is any tree then the ideal $I_T$ associated to the polytope $\cim_T$ can be obtained by taking iterative QIGs of the ideals of star trees, which are always the zero ideal. As a result we obtain an explicit square-free, quadratic Gr\"obner basis for $I_T$. 
These characteristic imset ideals serve as a first example of ideals which arise via iterated quasi-independence gluings, but not as codimension 0 toric fiber products (even up to change of coordinates). 
% \textcolor{joegreen}{We should think about this claim a little more. We might be able to make a dimension argument about the variety.} 
It would be interesting to find additional examples of ideals which arise in this way, similar to the many families of ideals which come from iterated toric fiber products \cite{ananiadi2021grobner, coons2021toric, cummings2021invariants,engstrom2014multigraded,rauh2016lifting,  sturmfels2005toric, sturmfels2008toric}. 

The remainder of the paper is organized as follows. In Section \ref{sec: preliminaries} we provide some background on DAGs, characteristic imsets, quasi-independence ideals, and toric fiber products. 
In Section~\ref{sec:QIG}, we introduce quasi-independence gluings and show that their Gr\"obner bases can be computed via lifting provided the original ideals satisfy a certain technical condition. 
In Section~\ref{section:CIMTree}, we apply quasi-independence gluings to characteristic imset ideals to derive the Gr\"obner basis of the ideal when the underlying graph is a tree. In Section~\ref{section:CIMCycle}, we discuss partial results for the characteristic imset ideal of a cycle. %give combinatorial and polyhedral constraints on the weight vector needed the produce a Gr\"obner basis for the path that can be used to construct a Gr\"obner basis for the cycle. %show that the characteristic imset ideal of a cycle also arises as a quasi-independence gluing. 
%---SECTION: Preliminaries---
\section{Preliminaries}
\label{sec: preliminaries}

%---SUBSECTION: DAGs and Imsets---
\subsection{Directed Acyclic Graphs and Characteristic Imset Polytopes}
\label{subsection:DAGandCIM}
Recall that given a directed acyclic graph (DAG) $\cg = ([n], E)$ and a collection of jointly distributed random variables $(X_1,\ldots, X_n)$ the model $\cm(\cg)$ consists of all densities $f(x_1,\ldots, x_n)$ such that
\[
f(x_1,\ldots, x_n) = \prod_{i=1}^nf(x_i | x_{\pa_\cg(i)}),
\]
where $\pa_\cg(i) = \{j \in [n] : j\rightarrow i\in E\}$ is the set of \emph{parents} of $i$ in $\cg$.  Any density $f$ that satisfies the above equation is said to be \emph{Markov} to the graph $\cg$. 

%DAG models play a fundamental role in modern data science and artificial intelligence through their applications in probabilistic inference \cite{koller2009probabilistic} and causality \cite{pearl2009causality}. 
%In such fields, the combinatorics of the graph can be used to determine complexity bounds for probabilistic inference algorithms \cite[Chapter 9]{koller2009probabilistic} and estimators for the causal effect of one variable in the system on another \cite[Chapter 3]{pearl2009causality}. 
%Given these applications, it is valuable to have a DAG representation of a data-generating distribution, especially if such a representation can be learned directly from data.
%This is the basic problem of \emph{causal discovery}, for which a variety of methods based on the combinatorics of DAGs have been proposed \cite{chickering2002optimal, kuipers2022efficient, LRS2022a, solus2017consistency, spirtes1991algorithm, tsamardinos2006max}. 

%More recently, efforts were made to rephrase the problem of causal discovery as a linear-programming problem by embedding DAGs as 0/1-vectors in real-Euclidean space \cite{studeny2010characteristic}: 

As previously discussed, many classic methods for causal discovery are based on the combinatorics of DAGs but recent work has shown that it can instead be seen as a linear program by embedding DAGs as 0/1-vectors in real-Euclidean space in the following way. 
\begin{definition}
Given a DAG $\cg = ([n], E)$, the \emph{characteristic imset} of $\cg$ is
\begin{equation*}
    \begin{split}
        c_\cg &: \{S\subseteq [n]: |S|\geq 2\} \longrightarrow \{0,1\}; \\
        c_\cg &: S \longmapsto
        \begin{cases}
            1   &   \mbox{ if there exists $i\in S$ such that $S\setminus\{i\}\subseteq\pa_\cg(i)$},\\
            0   &   \mbox{ otherwise}.
        \end{cases}
    \end{split}
\end{equation*} 
\end{definition}
Note that for a DAG $\cg$ with $n$ vertices we can realize $c_\cg$ as a 0/1-vector in $\rr^{2^n - n - 1}$. We then define the \emph{characteristic imset polytope} to be $\CIM_n = \conv(c_\cg : \cg \mbox{ a DAG with nodes $[n]$})$. 

An important feature of modeling with DAGs is that two different DAGs can encode the same set of densities: we say that $\cg$ and $\ch$ are \emph{Markov equivalent} if $\cm(\cg) = \cm(\ch)$.  
The phenomenon of Markov equivalence plays a key role in the problem of causal discovery when one wishes to learn a DAG representation based on observational data alone since, in this case, one can only hope to learn the DAG up to Markov equivalence. 
Hence, a variety of characterizations of Markov equivalence have been given, which include the following:
%---Theorem: Markov Equivalence---
\begin{theorem}
    \label{thm: markov}
    The following are equivalent:
    \begin{enumerate}
        \item $\cg = ([n], E)$ and $\ch = ([n], E^\prime)$ are Markov equivalent, 
        \item \cite{studeny2010characteristic} $c_\cg = c_\ch$, and
        \item \cite{verma2022equivalence} $\cg$ and $\ch$ have the same skeleton and v-structures.
    \end{enumerate}
\end{theorem}
Here, the \emph{skeleton} of a DAG $\cg = ([n], E)$ refers to its underlying undirected graph, and a \emph{v-structure} refers to a triple of vertices $i,j,k$ for which $i\rightarrow j, k\rightarrow j\in E$ but $i$ and $k$ not adjacent in $\cg$. 
In the following, we will refer to the graph $P_\cg$ given by undirecting all edges in $\cg$ that are not edges of a v-structure as the \emph{pattern} of $\cg$. 
An alternative graphical representation of the Markov equivalence class of $\cg$ is its so-called \emph{essential graph}, denoted $E_\cg$, which is a graph with both directed and undirected edges and the same skeleton as $\cg$ in which the directed edges are exactly those edges that are oriented in the same direction in all members of the Markov equivalence class of $\cg$.
Hence, Theorem~\ref{thm: markov} states that two DAGs are Markov equivalent if and only if they have the same characteristic imsets and the same pattern (or essential graph). See Figure~\ref{fig:patternsAndEssentialGraphs} for an explanation of the distinction between patterns and essential graphs. Given an undirected graph $G = ([n], E)$ we let $\meq(G)$ denote the set of patterns of all DAGs that have skeleton $G$. 

\begin{figure}
    \centering
\begin{tikzpicture}

%LEFT DAG
\begin{scope}[shift = {(-3,0)}]
\node at (-0.7,0.866) {1};
\node at (-0.7,-0.866) {2};
\node at (-0.3,0) {3};
\node at (1.3,0) {4};
\node at (1.7,0.866) {5};
\node at (1.7,-0.866) {6};

\node at (0.5,-1) {$\cg$};

\draw [-Stealth] (0,0)--(1,0);
\draw [rotate = 120, -Stealth] (1,0)--(0,0);
\draw [rotate = 240, Stealth-] (1,0)--(0,0);
\draw [shift = {(1,0)}, rotate = 60, -Stealth] (1,0)--(0,0);
\draw [shift = {(1,0)}, rotate = -60, Stealth-] (1,0)--(0,0);
%\draw [shift = {(1/2,0)}, -Stealth] (1,0.85) -- (1,-0.8);
\end{scope}

%CENTER PATTERN
\begin{scope}
\node at (0.5,-1) {$P_\cg$};

\draw [-Stealth] (0,0)--(1,0);
\draw [rotate = 120] (1,0)--(0,0);
\draw [rotate = 240] (1,0)--(0,0);
\draw [shift = {(1,0)}, rotate = 60, -Stealth] (1,0)--(0,0);
\draw [shift = {(1,0)}, rotate = -60] (1,0)--(0,0);
%\draw [shift = {(1/2,0)}, -Stealth] (1,0.85) -- (1,-0.85);
\end{scope}

%RIGHT ESSENTIAL GRAPH
\begin{scope}[shift = {(3,0)}]
\node at (0.5,-1) {$E_\cg$};

\draw [-Stealth] (0,0)--(1,0);
\draw [rotate = 120] (1,0)--(0,0);
\draw [rotate = 240] (1,0)--(0,0);
\draw [shift = {(1,0)}, rotate = 60, -Stealth] (1,0)--(0,0);
\draw [shift = {(1,0)}, rotate = -60,Stealth-] (1,0)--(0,0);
%\draw [shift = {(1/2,0)}, -Stealth] (1,0.85) -- (1,-0.85);
\end{scope}
\end{tikzpicture}

% \begin{tikzpicture}

% %A DAG
% \begin{scope}[shift = {(-3,0)}]
% \node at (-0.7,0.866) {1};
% \node at (-0.7,-0.866) {2};
% \node at (-0.3,0) {3};
% \node at (1.3,0) {4};
% \node at (1.7,0.866) {5};
% \node at (1.7,-0.866) {6};

% \node at (0.5,-1) {$\cg$};

% \draw [-Stealth] (0,0)--(1,0);
% \draw [rotate = 120, -Stealth] (1,0)--(0,0);
% \draw [rotate = 240, -Stealth] (1,0)--(0,0);
% \draw [shift = {(1,0)}, rotate = 60, Stealth-] (1,0)--(0,0);
% \draw [shift = {(1,0)}, rotate = -60, Stealth-] (1,0)--(0,0);
% \draw [shift = {(1/2,0)}, -Stealth] (1,0.85) -- (1,-0.8);
% \end{scope}

% %LEFT PATTERN
% \begin{scope}
% \node at (0.5,-1) {$P_\cg$};

% \draw (0,0)--(1,0);
% \draw [rotate = 120, -Stealth] (1,0)--(0,0);
% \draw [rotate = 240, -Stealth] (1,0)--(0,0);
% \draw [shift = {(1,0)}, rotate = 60] (1,0)--(0,0);
% \draw [shift = {(1,0)}, rotate = -60] (1,0)--(0,0);
% \draw [shift = {(1/2,0)}, -Stealth] (1,0.85) -- (1,-0.85);
% \end{scope}

% %RIGHT ESSENTIAL GRAPH
% \begin{scope}[shift = {(3,0)}]
% \node at (0.5,-1) {$E_\cg$};

% \draw [-Stealth] (0,0)--(1,0);
% \draw [rotate = 120, -Stealth] (1,0)--(0,0);
% \draw [rotate = 240, -Stealth] (1,0)--(0,0);
% \draw [shift = {(1,0)}, rotate = 60, Stealth-] (1,0)--(0,0);
% \draw [shift = {(1,0)}, rotate = -60, Stealth-] (1,0)--(0,0);
% %\draw plot [shift = {(1/2,0)}, smooth, tension=2, -Stealth] coordinates { (1,0.85) (1.2,0) (1,-0.85)};
% \draw [shift = {(1/2,0)}, -Stealth] (1,0.85) -- (1,-0.85);
% \end{scope}
% \end{tikzpicture}
    \caption{A DAG $\cg$ and the associated pattern $P_\cg$ and essential graph $E_\cg$. The edge $4 \rightarrow 6$ is part of no v-structures, so it is undirected in the pattern. Since $3 \rightarrow 4 \leftarrow 5$ forms a v-structure, all DAGs in the Markov equivalence class of this DAG must have the same orientation $4 \rightarrow 6$, which appears in the essential graph.}
    %\caption{A DAG $\cg$ and the associated pattern $P_\cg$ and essential graph $E_\cg$. The edge $3 \rightarrow 4$ is part of no v-structures, so it is undirected in the pattern. Since $1 \rightarrow 3 \leftarrow 2$ forms a v-structure, all DAGs in the Markov equivalence class of this DAG must have the same orientation $3 \rightarrow 4$, which appears in the essential graph.}
    \label{fig:patternsAndEssentialGraphs}
\end{figure}
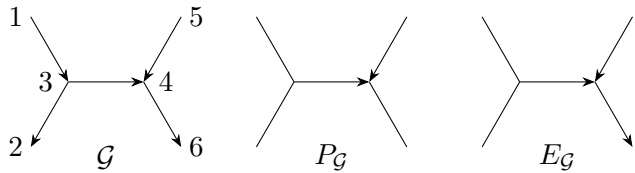

Due to its relevance in the problem of causal discovery, the polyhedral geometry of $\CIM_n$ has been studied.  
In \cite{studeny2017towards, studeny2021dual, xi2015characteristic} families of facets of $\CIM_n$ and certain subpolytopes of $\CIM_n$ are identified. Similar results have also been found for families of edges in \cite{LRS2022a, LRS2022b}. 
In \cite[Proposition 2.5]{LRS2022a}, it is noted that for an undirected graph $G = ([n], E)$
\[
\CIM_G = \conv(c_\cg : \cg \mbox{ has skeleton $G$})
\]
is a face of $\CIM_n$.  
In \cite{LRS2022b}, a complete characterization of the edges of $\CIM_G$ for $G$ a tree is also given. 
The following useful fact was also observed:
\begin{lemma}
    \cite[Proposition 1.9]{LRS2022b}\label{lem: CIMP of a star graph}
    If $G = ([n], E)$ is the \emph{star graph}; i.e., $E = \{j-i^* : j\in [n]\setminus\{i^*\}\}$ for some $i^*\in[n]$, then $\CIM_G$ is a simplex. 
\end{lemma}
It is also not difficult to see that $\CIM_G$ is a simplex when $G$ is the path graph on four vertices. 
We will use these observations in the coming sections. 

\subsection{Gr\"obner bases for Toric Ideals from Bipartite Graphs and Quasi-Independence}
In this section we provide some background on toric ideals. We also describe those that arise from quasi-independence models in statistics or equivalently toric ideals arising from bipartite graphs. In the following sections, we will use quasi-independence ideals to study the toric ideals associated to characteristic imset polytopes. 

Quasi-independence models are log-linear models for the joint distribution of two discrete random variables $X$ and $Y$, with respective state spaces $[r]$ and $[s]$, where some subset of the states of $X$ and $Y$ are forbidden from occurring together. These forbidden states of the joint distribution are called \emph{structural zeroes} of the model. These models are typically specified by a set $Q \subseteq [r] \times [s]$ which consists of the allowed states of the joint distribution. The structural zeros of the model are then all $(j, k) \notin Q$.

% \textcolor{joegreen}{J: Maybe we should change the letter for our set here. We have used $S$ to mean literally every kind of set in this paper. Let's vote!
% \begin{enumerate}
% \item $S$-homogeneous
% \item $Q$-homogeneous
% \item \u{G}-homogeneous (appropriating this cool $G$ from Turkish, don't know how to pronounce it)
% \item Some other letter (be specific)
% \end{enumerate}
% }

% \textcolor{joegreen}{J: I vote for \#2}
% \textcolor{blue}{B: I also vote for 2}
% \textcolor{orange}{IP: Although it would have been lovely to hear people trying to pronounce \#3, I need to hold back and say \#2.}
% \textcolor{red}{L: 2 it is!  Although my real vote is for the one no one can pronounce... technically thats a tie.  Just saying. :D}

\begin{definition}
\label{defn:quasiIndep}
Let $r$ and $s$ be two positive integers and $Q \subseteq [r] \times [s]$. The \emph{quasi-independence ideal} associated to the set $Q$ is the kernel of the monomial map
\begin{align*}
    \phi_Q: \kk[z_{jk} ~|~ (j, k) \in Q]& \to \kk[x_j, y_k ~|~ j \in [r] , k \in [s]]\\
            z_{jk}& \mapsto x_j y_k
\end{align*}
and is denoted by $I_Q = \ker(\phi_Q)$. 
\end{definition}
Note that in the definition of the quasi-independence model we have simply omitted coordinates which correspond to structural zeroes. Also observe that $I_Q$ is a toric ideal since it is the kernel of a monomial map. Moreover, Theorem~\ref{thm:gbofquasi-independence} gives an explicit generating set for $I_Q$.

We frequently associate an $r \times s$ matrix $A_Q = \left(a_{jk}\right)_{(j,k) \in [r] \times [s]}$ to the set $Q$ given by
\[a_{jk} = \begin{cases} z_{jk} &\mbox{if}~ (j,k) \in Q \\ 0 &\mbox{otherwise} \end{cases}\]
Such a set $Q$ also has a naturally associated bipartite graph. We define $G_Q \subseteq K_{r,s}$ to be the bipartite graph on two disjoint and independent vertex sets $[r]$ and $[s]$ with edge set $Q$. The structural zeros encode the absence of an edge in $G_Q$. This graph is usually non-planar, and so in examples and figures we will typically consider $A_Q$. However, $G_Q$ establishes a link to a different viewpoint. The quasi-independence ideals are also studied under the title of toric ideals of edge rings \cite{villarreal} and many combinatorial and algebraic aspects of these ideals are known \cite{Herzog2018}. Additionally the so-called edge polytopes, whose vertices are the columns of the incidence matrix of $G_Q$, are introduced in \cite{de1995grobner} and their faces are characterized in terms of certain subgraphs of $G_Q$ \cite[Theorem 2.18]{por21}. 
%\textcolor{red}{In particular, $\CIM_G$ is an edge polytope of a tree, when $G$ is the star graph.}
\begin{theorem}
\label{thm:gbofquasi-independence}
Let $Q \subseteq [r] \times [s]$ and let $j_1 - k_1, j_2 - k_1 , j_2 - k_2,\dots, j_\ell - j_\ell, j_1 - k_\ell $ be the edges of an induced cycle of $G_Q$. Then
\[\prod\limits_{i=1}^\ell z_{j_i,k_i} - \prod\limits_{i=1}^\ell z_{j_{i+1},k_i}\]
(where $i$ is taken modulo $\ell$) is a binomial of $I_Q$. Moreover, the set of all such binomials forms the universal Gr\"obner basis for $I_Q$.
\end{theorem}

A proof of Theorem \ref{thm:gbofquasi-independence} for general directed graphs is given in \cite[Proposition 4.3]{GITLER2010430}. Assuming that the edges of $G_Q$ are directed from the vertices in $[r]$ to the vertices in $[s]$, we may apply this proof. For completeness we give an elementary proof by computing the normal form of all $S$-polynomials.

\begin{proof}
The multisets of $j_i$ and $k_i$ in the two monomials are the same, and so every binomial of this form lies in $I_Q$. Let $f$ and $g$ be binomials associated to cycles. We form the $S$-polynomial $S(f,g)$ by multiplying $f$ and $g$ by monomials, so that their leading terms are the same, and subtracting to cancel the leading terms. After scaling $f$ and $g$, all four monomials have the same multisets of $j_i$ and $k_i$, and so this holds for $S(f,g)$ as well.

We form a bipartite graph $G'$ with vertex sets $[r]$ and $[s]$ and edges $j - k$ for $z_{jk}$ appearing in $S(f,g)$ (with multiple edges allowed between two vertices). Color the edges by which monomial they appear in. Traverse the edges of the graph, alternating between the two edge colors. Because both monomials contain the same multiset of $j_i$ and $k_i$, each vertex is part of the same number of each color of edge. Consequently each vertex entered has an edge of the opposite color that one can exit through, so we can traverse the edges of the graph until a cycle is found. The leading term of the associated cycle binomial divides some (not necessarily leading) term of $S(f,g)$. 
Note that such a cycle binomial always exists since every time you enter a vertex on an edge of one color there is always an edge of the other color that you can exit along.  
Hence, you can traverse the graph in this way until you return to a previously visited vertex, at which point you must have discovered such a cycle binomial.
% {\color{purple} why must there always exist a cycle in G'?}\textcolor{joegreen}{Because whenever you enter a vertex via an edge of one color, there is an edge of the other to exit through. You can keep traversing edges until you reach a vertex that was previously visited.}

With binomial ideals, we can perform the division algorithm even if the leading term of the cycle binomial divides the trailing term. This is because the division algorithm replaces one monomial with another when we divide by a binomial, so the number of terms never increases. Via the previous argument, at each step in the division algorithm, one can obtain a cycle in $G'$ and use it to decrease some monomial in the term order. 
This process can only terminate at the normal form 0.
\end{proof}

\begin{example}
Let $r = s = 5$ and consider $Q = \left([1,2] \times [1,4]\right) \cup \left( [2,5] \times [4,5]\right)$. The associated matrix $A_Q$ is
\[\begin{bmatrix} z_{11} & z_{12} & z_{13} & z_{14} & 0 \\
z_{21} & z_{22} & z_{23} & z_{24} & z_{25} \\
0 & 0 & 0 & z_{34} & z_{35} \\
0 & 0 & 0 & z_{44} & z_{45} \\
0 & 0 & 0 & z_{54} & z_{55}
\end{bmatrix}.\]
The graph $G_Q$ has 12 cycles of length 4, which correspond to the $2 \times 2$ submatrices with no structural zeroes. There are no larger cycles in $G_Q$, so $I_Q$ is generated by 12 quadratic binomials.% A generating set for the kernel of $\phi_Q$ is
%\[\begin{array}{ccc}
%     z_{11}z_{22} - z_{12}z_{21}, &z_{11}z_{23} - z_{13}z_{21}, &z_{11}z_{24} - z_{14}z_{21}, \\
%     z_{11}z_{22} - z_{12}z_{21}, &z_{11}z_{23} - z_{13}z_{21}, &z_{11}z_{24} - z_{14}z_{21},\\
%     z_{42}z_{53} - z_{52}z_{43},& z_{42}z_{54} - z_{44}z_{52}, &z_{42}z_{55} - z_{45}z_{52} \\
%     z_{43}z_{54} - z_{44}z_{53}, &z_{43}z_{55} - z_{45}z_{53}, &z_{44}z_{55} - z_{45}z_{54}
%\end{array}\]
\end{example}

Next, we state when the quasi-independence ideal is generated by quadratic square-free binomials. A bipartite graph $G$ is called \emph{chordal bipartite} if it contains no induced cycles with at least 6 vertices. The graph $G_Q$ above is an example of such a graph. We have the following special case of \cite[Theorem 1.2]{OHSUGI1999509}.
\begin{theorem}\label{thm: quadratic and square-free condition for quasi-independence ideals}
Let $Q \subseteq [r] \times [s]$. Then $I_Q$ is generated by quadratic square-free binomials if and only if $G_Q$ is chordal bipartite.
\end{theorem}

In Section~\ref{sec:QIG} we define a notion of iterated gluing by factoring through quasi-independence maps. Factoring through maps associated to chordal bipartite graphs preserves the property that the ideal is generated by square-free quadratics.

% \begin{theorem}[{\cite[Theorem 1.2]{OHSUGI1999509}}]\label{thm: quadratic and square-free condition for quasi-independence ideals}
% Let $C$ be an even cycle of length $\geq 6$ on the bipartite graph $G_Q$. The quasi-independence ideal is generated by quadratic square-free binomials if and only if $C$ has an
% even-chord or $C$ has three odd-chords $e$, $e'$, $e''$ such that $e$ and $e'$ cross in $C$.
% \end{theorem}
% In particular, for a chordal bipartite graph $G_Q$, its quasi-independence ideal is generated by quadratic square-free binomials. 

\subsection{Toric Fiber Products}
\label{subsec: toric fiber products}
In this section we provide a brief outline of the toric fiber product as the contraction of an ideal under a monomial homomorphism. Unlike the usual description of the toric fiber product, we reframe the theory slightly into the language of quasi-independence maps.

Let $r$ be a positive integer and for each $i \in r$, let $s_i$ and $t_i$ be positive integers. Let $\mathcal{A} = \{a_1,\dots,a_r\} \subset \mathbb{Z}^d$ for some $d$. We define two graded polynomial rings:
\[\kk[x] := \kk[x_j^i~|~i \in [r], j \in [s_i]] \quad \text{and} \quad \kk[y] := \kk[y_k^i~|~i \in [r], k \in [t_i]]\]
where $\deg(x_j^i) = \deg(y_k^i) = a_i$. Throughout this subsection let $I \subseteq \kk[x]$ and $J \subseteq \kk[y]$ be homogeneous ideals with respect to the multigrading. We define a quasi-independence map associated to the multigrading. Let
\[Q(\mathcal{A}) = \bigsqcup_{i=1}^r \left(\{i\} \times [s_i] \right) \times \left( \{i\} \times [t_i]\right)\]
Let $\kk[z] = \kk[z_{jk}^i:\left((i,j),(i,k)\right) \in Q(\mathcal{A})]$. The quasi-independence map $\phi_{Q(\mathcal{A})}:\kk[z] \to \kk[x,y]$ is
\[z_{jk}^i \mapsto x_j^i y_k^i.\]
Morally, we glue $x_j^i$ and $y_k^i$ together when they have the same degree in the multigrading $\ca$. Note that each subset $\left(\{i\} \times [s_i] \right) \times \left( \{i\} \times [t_i]\right) \subseteq Q(\mathcal{A})$ is the edge set of a complete bipartite subgraph of $G_Q$. Up to reordering the rows and columns of $A_Q$, we see that $A_Q$ is block diagonal with blocks indexed by $i \in [r]$. In this subsection we will only consider block diagonal $A_Q$, but in Section~\ref{sec:QIG} we extend the following definitions and results to more general quasi-independence maps.

\begin{definition}
The \emph{toric fiber product} of the ideals $I$ and $J$ is
\[I \times_{\mathcal{A}} J = \phi_{Q(\mathcal{A})}^{-1}(I+J).\]
\end{definition}

It is frequently difficult to compute a Gr\"obner basis of $I \times_\mathcal{A} J$, even with Gr\"obner bases for $I$ and $J$. However, if $\mathcal{A}$ is a set of linearly independent vectors, a Gr\"obner basis can be computed. Let 
\[
f = \sum\limits_{\ell} c_\ell x_{j_{\ell,1}}^{i_{\ell,1}} x_{j_{\ell,2}}^{i_{\ell,2}} \cdots x_{j_{\ell,d}}^{i_{\ell,d}} \in \kk[x].
\]
%\textcolor{red}{L: This notation is comically hard to look at/interpret... can we fix it without destroying the paper? Really what it is supposed to mean is that for each $\ell$ we picked a multiset $I_\ell$ of $i$'s and a multiset $J_\ell$ of $j$'s.  If this was a random sample ${\bf X}_1,\ldots, {\bf X}_n$ from a multivariate distribution ${\bf X} = (X_1,\ldots, X_d)$ I would typically denote the entries in each draw as ${\bf X}_i = (X_{1,i},\ldots, X_{d,i})$ or ${\bf X}_i = (X_{i,1},\ldots, X_{i,d})$... perhaps the latter is more natural in a data science setting where one typically treats the columns of a data frame as the variables and the rows as individual samples.  Does notation such as $x_{j_{\ell,1}}^{i_{\ell,1}} x_{j_{\ell,2}}^{i_{\ell,2}} \cdots x_{j_{\ell,d}}^{i_{\ell,d}}$ bother people?  does it look a little better?}
If $f$ is homogeneous with respect to a linearly independent multigrading, then the multiset of upper indices in each monomial must be the same. 
After reindexing the variables in each monomial, we can suppress $\ell$ from the upper index.
\[f = \sum\limits_{\ell} c_\ell x_{j_{\ell,1}}^{i_1} x_{j_{\ell,2}}^{i_2} \cdots x_{j_{\ell,d}}^{i_d}\]
Let $k = (k_1,\ldots,k_d) \in [t_{i_1}] \times [t_{i_2}] \times \ldots \times [t_{i_d}]$. The \emph{lift} of $f$ by $k$ is 
\[
f_k = \sum\limits_{\ell} c_\ell z_{j_{\ell,1} k_1}^{i_1} z_{j_{\ell,2} k_2}^{i_2} \cdots z_{j_{\ell,d} k_d}^{i_d} \in \kk[z].
\]
The set of all lifts of $f$ is denoted $\text{Lift}(f)$. For a set of polynomials $F \subset I$, we denote the set of all lifts of elements in $F$ by $\text{Lift}(F)$. Similarly we define lifts of elements in $\kk[y]$.

Let $H_{Q(\mathcal{A})}$ denote the universal Gr\"obner basis of the quasi-independence ideal defined in Theorem~\ref{thm:gbofquasi-independence}. Let $B_{Q(\mathcal{A})}$ denote the matrix of the exponent vectors of $\phi_{Q(\mathcal{A})}$. For an ideal $I \subseteq \kk[x]$, we say a finite set $G \subseteq I$ is a \emph{pseudo-Gr\"obner basis} for $I$ with respect to weight $\omega$ if $\langle \init_{\omega}(G) \rangle = \init_{\omega}(I)$. Note that this differs from a Gr\"obner basis in that here $\init_{\omega}(G)$ may not be a set monomials.
Sullivant proved the following theorem.

\begin{theorem}\cite[Theorem 12]{sullivant2007toric}
Let $\mathcal{A}$ be a linearly independent multigrading. Let $F$ and $G$ be Gr\"obner bases for $I$ and $J$ with respect to weight order $\omega_1$ and $\omega_2$ respectively. Then
\[\lift(F) \cup \lift(G) \cup H_{Q(\mathcal{A})}\]
forms a pseudo-Gr\"obner basis of $I \times_{\mathcal{A}} J$ with respect to the weight order $(\omega_1^T,\omega_2^T) B_{Q(\mathcal{A})}$.
\end{theorem}
After perturbing the term order, this set forms a Gr\"obner basis. In algebraic statistics, the toric fiber product is a commonly used tool for computing a Gr\"obner basis of an ideal with an associated combinatorial structure \cite{ananiadi2021grobner, coons2021toric, cummings2021invariants,engstrom2014multigraded,rauh2016lifting,  sturmfels2005toric, sturmfels2008toric}. Typically the informal gluing of variables in the description of the toric fiber product corresponds to a gluing of combinatorial structures such as trees or simplicial complexes. The toric fiber product allows us to iteratively compute a Gr\"obner basis as we iteratively glue combinatorial structures.

\section{Quasi-Independence Gluings}
\label{sec:QIG}

In this section we introduce a new generalization of the toric fiber product which we call a \emph{quasi-independence gluing} (QIG).
Similar to the toric fiber product, QIGs are the contraction of an ideal under a monomial homomorphism, so many of the techniques used to prove results for the toric fiber product can be naturally extended to prove similar results for QIGs. In particular, we describe the Gr\"obner basis (or generating set) of the QIG of two ideals by lifting the Gr\"obner bases of the original ideals. The main difference between the two operations is that QIGs allow for contraction with respect to the monomial parameterization  of any quasi-independence model whereas toric fiber products can be realized as block-diagonal QIGs. Our notation and terminology for this section is adapted from \cite{aoki2012markov, coons2021quasi, sullivant2007toric}. Throughout this section we fix $r,s \in \mathbb{Z}_{>0}$ and
\[\kk[x] = \kk[x_j:j\in[r]] \quad \text{and} \quad \kk[y] = \kk[y_k: k \in [s]].\]
We begin with the following definition.

\begin{definition}
Let $\omega \in \mathbb{R}^r$ be a weight. Let $Q \subseteq [r] \times [s]$ and $f \subset \kk[x]$ be a polynomial which is homogeneous with respect to total degree. 
Let $\init_\omega(f) = x_{j_1} x_{j_2} \ldots x_{j_d}$.  We say $f$ is \emph{weakly $Q$-homogeneous} with respect to $\omega$ if for every monomial $m' = x_{j_1'} x_{j_2'} \ldots x_{j_d'}$ of $f$ there exists a permutation $\sigma$ of the elements of $[d]$ such that%re-ordering of the $j_\ell'$ such that
%\textcolor{blue}{clarify what the reordering is doing}{\color{orange}I think explaining in Example 3.9 would be helpful.}
\[
\{ 
(k_1, \ldots, k_d) \in [s]^d ~|~ (j_\ell, k_\ell) \in Q, 1 \leq \ell \leq d \}
\subseteq
\{ 
(k_1, \ldots, k_d) \in [s]^d ~|~ (j'_{\sigma(\ell)}, k_\ell) \in Q, 1 \leq \ell \leq d \}. 
\]
\end{definition}

We typically omit the $\sigma$ when describing this reordering since the specific permutation is not important. Later we define and motivate a notion of \emph{strong $Q$-homogeneity}. Note that while weak $Q$-homogeneity does require homogeneity with respect to total degree, it does not necessarily correspond to homogeneity with respect to a multigrading. 

\begin{definition}\label{def: QIG}
Let $Q \subseteq [r] \times [s]$ and $I \subseteq \kk[x]$ and $J \subseteq \kk[y]$ be homogeneous ideals. The \emph{quasi-independence gluing} of $I$ and $J$ with respect to the set $Q$ is
\[
I \times_Q J := \phi_Q^{-1}(I + J).
\]
\end{definition}

Recall that defining the toric fiber product required a set $Q(\mathcal{A})$, defined in terms of a multigrading $\mathcal{A}$, and that $A_{Q(\mathcal{A})}$ is block diagonal. So QIGs include toric fiber products but allow for more general structures.

%The intuitive interpretation of a quasi-independence gluing is that the set $Q$ contains the pairs of variables $(x_j, y_k)$ which are allowed to be glued together to obtain a new variable $z_{jk}$. This is analogous to how variables with the same degree are glued together in a toric fiber product but allows for more general structures. Recall that the toric fiber product is the ideal $I \times_\ca J = \phi_{Q(\mathcal{A})}^{-1}(I+J)$ where
%\begin{align*}
%    \phi_{Q(\mathcal{A})}: \kk[z_{jk}^i]& \to \kk[x_j^i, y_k^i | i \in [r], j \in [s_i], k \in [t_i] ]\\
%            z_{jk}^i & \mapsto x_j^i y_k^i
%\end{align*}

%Observe that $\phi_{Q(\mathcal{A})}$ is exactly the quasi-independence map corresponding to the set $Q = \bigsqcup_{i = 1}^r [s_i] \times [t_i]$. In other words, if one considers the matrix $A_Q$ 
%then toric fiber products are quasi-independence gluings where the matrix $A_Q$ is block-diagonal. So QIGs include toric fiber products but allow for more general gluing structures. 

In this section, we show that when $I$ and $J$ are weakly $Q$-homogeneous, we can construct a Gr\"obner basis for $I \times_Q J$ from Gr\"obner bases of $I$ and $J$. Our approach is essentially the same as that used to determine Gr\"obner bases for toric fiber products. Because of this we try to use the same notation and present our argument in the same format whenever possible. We begin with the following two lemmas from \cite{sullivant2007toric}.

\begin{lemma}[{\cite[Lemma 2.2]{sullivant2007toric}}]
\label{lemma:initIdealContainment}
Let  $\phi: \kk[z_1, \ldots, z_n] \to \kk[t_1, \ldots t_d]$ be a monomial map with a matrix of exponent vectors $B$. If $I$ is an ideal in $\kk[t]$ and $\omega \in \zz_{\geq 0}^d$ is a weight vector on $\kk[t]$, then
\[
\init_{\omega^T B}(\phi^{-1}(I)) \subseteq \phi^{-1}(\init_\omega(I)). 
\]
\end{lemma}

\begin{lemma}[{\cite[Lemma 2.3]{sullivant2007toric}}]
\label{lemma:monomialIdealContraction}
Let $M = \langle m_1, \ldots m_r \rangle \subset \kk[t]$ be a monomial ideal. Then 
\[
\phi^{-1}(M) = \phi^{-1}(\langle m_1 \rangle) + \ldots + \phi^{-1}(\langle m_r \rangle). 
\]
Furthermore, $\phi^{-1}(M) = M' + \ker(\phi)$ where $M'$ is a monomial ideal. 
\end{lemma}

%\textcolor{joegreen}{I changed the outline of the strategy. I feel like this makes it clearer, but if you like the other way better, I just commented it out.}

As explained in \cite{sullivant2007toric}, these two lemmas suggest a potential strategy for determining $\phi^{-1}(I)$:

%As explained in \cite{sullivant2007toric}, these two lemmas suggest a potential strategy for determining $\phi^{-1}(I)$ which is to first determine $\phi^{-1}(\init_\omega(I))$ and then use this to find a set of polynomials $G \subseteq \phi^{-1}(I)$ such that $\langle in_{\omega^T B}(G) \rangle = \init_{\omega^T B}(\phi^{-1}(I))$ which implies $G$ is a pseudo-Gr\"obner basis for $\phi^{-1}(I)$ with respect to the weight order given by $\omega^T B$. We will follow this strategy in order to determine a Gr\"obner basis for $I \times_Q J$. We begin by proving the following lemma which is analogous to \cite[Lemma 2.5]{sullivant2007toric} for QIGs and has the exact same proof. 

\begin{enumerate}
    \item Determine $\phi_B^{-1}(\init_\omega(I))$.
    \item Find a candidate Gr\"obner basis $G \subseteq \phi^{-1}(I)$, so that by Lemma~\ref{lemma:initIdealContainment} we have \[\init_{\omega^T B}(\langle G \rangle) \subseteq \init_{\omega^T B}(\phi^{-1}(I)) \subseteq \phi^{-1}(\init_\omega(I)).\]
    \item Show that $\init_{\omega^T B}(\langle G \rangle) = \phi^{-1}(\init_\omega(I))$.
    \item Deduce that $G$ is a Gr\"obner basis for $\phi^{-1}(I)$ since
    \[\init_{\omega^T B}(\langle G \rangle) = \init_{\omega^T B}(\phi^{-1}(I)) = \phi^{-1}(\init_\omega(I)).\]
\end{enumerate}
We will follow this strategy in order to determine a Gr\"obner basis for $I \times_Q J$. We begin by proving the following lemma which is analogous to \cite[Lemma 2.5]{sullivant2007toric} for QIGs and has the exact same proof. 

\begin{lemma}
\label{lemma:singleMonomialContraction}
Let $Q \subseteq [r] \times [s]$ and let $m = x_{j_1} x_{j_2} \ldots x_{j_d}$ be a monomial in $\kk[x, y]$. Then
\[
\phi_Q^{-1}(\langle m \rangle) = \langle z_{j_1 k_1} z_{j_2 k_2} \ldots z_{j_d k_d} ~|~ k_1, \ldots k_d \in [s],  (j_\ell, k_\ell) \in Q \rangle + I_Q.
\]
Similarly, if $n = y_{k_1} y_{k_2} \ldots y_{k_d}$ is a monomial in $\kk[x, y]$. Then
\[
\phi_Q^{-1}(\langle n \rangle) = \langle z_{j_1 k_1} z_{j_2 k_2} \ldots z_{j_d k_d} ~|~ j_1, \ldots j_d \in [r],  (j_\ell, k_\ell) \in Q \rangle + I_Q.
\]
\end{lemma}
\begin{proof}
By symmetry it suffices to prove the first statement. 
By Lemma  \ref{lemma:monomialIdealContraction} we have that $\phi_Q^{-1}(\langle m \rangle) = M' + I_Q$ for some monomial ideal $M'$ which we claim is the ideal
\[
M' = \langle z_{j_1 k_1} z_{j_2 k_2} \ldots z_{j_d k_d} ~|~ k_1, \ldots k_d \in [n],  (j_\ell, k_\ell) \in Q \rangle. 
\]
Consider a monomial $m' = z_{j_{i_1} k_{i_1}} \ldots z_{j_{i_\ell} k_{i_\ell}}$ and note that 
$\phi_Q(m') = x_{j_{i_1}} y_{k_{i_1}} \ldots x_{j_{i_\ell}} y_{k_{i_\ell}}$ which belongs to $\langle m \rangle$ if and only if $x_{j_{i_1}} \ldots x_{j_{i_\ell}}$ belongs to $\langle m \rangle$. If $x_{j_{i_1}} \ldots x_{j_{i_\ell}} \in \langle m \rangle$ then $x_{j_1} x_{j_2} \ldots x_{j_d}$ divides $x_{j_{i_1}} \ldots x_{j_{i_\ell}}$ but this immediately implies that there exists a monomial in $M'$ dividing $m'$. 
\end{proof}

We now discuss how to lift Gr\"obner bases for the ideals $I$ and $J$ to a Gr\"obner basis for $I \times_Q J$. This once again is similar to the lifting presented in \cite{sullivant2007toric} (or subsection~\ref{subsec: toric fiber products}), but the lifting procedure is determined by $Q$ instead of a multigrading. 

Let $\omega$ be a weight and consider a weakly $Q$-homogeneous polynomial $f \in I$ of the form
\[
f = \sum_{\ell} c_\ell x_{j_{\ell,1}} x_{j_{\ell,2}}  \cdots x_{j_{\ell,d}}.
\]
Let $\init_\omega(f) = c_* x_{j_{*,1}} x_{j_{*,2}}  \cdots x_{j_{*,d}}$ and let $k = (k_1, \ldots, k_d) \in [s]^d$ such that the $k_i$ satisfy $(j_{*,i}, k_i) \in Q$. Note that since $f$ is weakly $Q$-homogeneous, we have that $(j_{\ell,i}, k_i) \in Q$ for all $\ell$ after rearranging each monomial $x_{j_{\ell,1}} x_{j_{\ell,2}} \ldots x_{j_{\ell,d}}$ appropriately.  

\begin{definition}
\label{defn:lift}
Let $Q \subseteq [r] \times [s]$
A \emph{lift} of $f$ by lower indices $k = (k_1,\dots,k_d)$ is a homogeneous polynomial $f_k$ of the form
\[
f_k = \sum_{\ell} c_\ell z_{j_{\ell,1} k_1} z_{j_{\ell,2} k_2} \cdots z_{j_{\ell,d} k_d}. 
\]
Furthermore, we denote the set of all possible lifts of $f$ by $\lift(f)$ which is the set
\[
\lift(f) = 
\{
f_k ~|~ k \in [s]^d,~ (j_{*,i}, k_i) \in Q,~ 1 \leq i \leq d \}
\}
\]
\end{definition}

\noindent Observe that for any polynomial $f_k \in \lift(f)$ it holds that $f_k \in I \times_Q J$ since
\[
\phi_Q(f_k) = y_{k_1} y_{k_2} \ldots y_{k_d} f \in I. 
\]
Now given a collection $F \subseteq I$ of weakly $Q$-homogeneous polynomials, let
\[
\lift(F) = \bigcup_{f \in F} \lift(f)
\]
and for a collection $G \subseteq J$ let $\lift(G)$ be the analogous collection. In terms of lifting, weak $Q$-homogeneity means that whenever the leading monomial of $f$ lifts by $k$, there exists lifts of all non-leading monomials of $f$ by $k$ as well.
 
\begin{theorem}
\label{thm:quasiIndepPseudoGrobner}
Let $F \subseteq I$ be a weakly $Q$-homogeneous Gr\"obner basis for $I$ with respect to the weight $\omega_1$ and $G \subseteq J$ be a weakly $Q$-homogeneous Gr\"obner basis  for $J$ with respect to the weight $\omega_2$. Then
\[
\lift(F) \cup \lift(G) \cup H_Q
\]
is a pseudo-Gr\"obner basis for $I \times_Q J$ with respect to the weight $\omega^T B_Q$
where $B_Q \in \zz^{(r+s) \times \# Q}$ is the matrix of exponents of the map $\phi_Q$ and $\omega = (\omega_1, \omega_2)$. 
\end{theorem}

\begin{proof}
%This proof follows the exact same structure as the proof of Theorem 2.8 of \cite{sullivant2007toric}. 
First recall that by Lemma \ref{lemma:initIdealContainment}, we have that
\[
\init_{\omega^T B_Q}(\phi_Q^{-1}(I+J)) \subseteq \phi_Q^{-1}(\init_\omega(I+J)). 
\]
We have already shown that $\lift(F) \cup \lift(G) \cup H_Q \subseteq I \times_Q J = \phi_Q^{-1}(I+J)$ so if we show that the leading terms of $\lift(F) \cup \lift(G) \cup H_Q$ form a pseudo-Gr\"obner basis for $\phi_Q^{-1}\left(\init_{(\omega_1,\omega_2)}(I+J)\right)$ %is actually pseudo-Gr\"obner basis for $\phi_Q^{-1}(\init_\omega(I+J))$
then we have that
\[
\init_{\omega^T B_Q}(\phi_Q^{-1}(I+J)) = \phi_Q^{-1}(\init_\omega(I+J))
\]
and thus $\lift(F) \cup \lift(G) \cup H_Q$ is a pseudo-Gr\"obner basis for $\phi_Q^{-1}(I+J) = I \times_Q J$.

Let $M = \init_\omega(I+J)$ and recall that by Lemma \ref{lemma:monomialIdealContraction} we only need to show that any monomial in $\phi_Q^{-1}(M)$ is divisible by the $\omega^T B_Q$-leading term of a polynomial in $\lift(F) \cup \lift(G)$. So suppose $m' = z_{j_1' k_1'} z_{j_2' k_2'} \ldots z_{j_t' k_t'}$ is a monomial in $\phi_Q^{-1}(M)$. Now observe that $F \cup G$ is a Gr\"obner basis for $I+J$ with respect to $\omega$ since $I$ and $J$ are in disjoint sets of variables and that any minimal generator of $M$ is either in $\kk[x]$ or $\kk[y]$ for the same reason.
So by Lemma \ref{lemma:singleMonomialContraction} $m'$ is divisible by a monomial $ m = z_{j_1 k_1} z_{j_2 k_d} \ldots z_{j_d k_d}$ and there exists some $f \in F$ or some $g \in G$ such that $m \in \lift(\init_\omega(f))$ or $m \in \lift(\init_\omega(g))$. 

We now suppose that $m \in \lift(\init_\omega(f))$ since the proof of the other case is the same. So $m$ is the lift of $in_w(f)$ for some valid tuple $k \in [s]^d$. But then $f_k \in \lift(f)$ since $f$ is weakly $Q$-homogeneous and of course $m = \init_{\omega^T B_Q}(f_k)$. Thus we have that $m'$ is divisible by the $\omega^T B_Q$-leading term of a polynomial in $\lift(F)$ or $\lift(G)$ which completes the proof.
\end{proof}

\begin{corollary}
\label{corollary:quasiIndepGrobner}
With the same assumptions as Theorem \ref{thm:quasiIndepPseudoGrobner}. Let $\omega$ be a weight vector such that $H_Q$ is a Gr\"obner basis for $I_Q$. Then there exists $\epsilon > 0$ such that
\[
\lift(F) \cup \lift(G) \cup H_Q
\]
is a Gr\"obner basis for $I \times_Q J$ with respect to the weight $\omega' = (\omega_1, \omega_2)^T B_Q + \epsilon \omega$. 
\end{corollary}
\begin{proof}
By Theorem \ref{thm:gbofquasi-independence}, there exists $\omega$ such that $H_Q$ is a Gr\"obner basis for $I_Q$. For sufficiently small $\epsilon$, $\omega'$ specifies the same leading terms of $\lift(F)$ and $\lift(G)$ while additionally determining leading terms of polynomials in $H_Q$. 
%The same proof used for \cite[Theorem 2.9]{sullivant2007toric} holds here. 
\end{proof}

The following example illustrates how one can use Theorem \ref{thm:quasiIndepPseudoGrobner} to find a generating set for a quasi-independence gluing. 

\begin{example}
\label{example:weaklySHomogeneousLifting}
Consider the following monomial map:
\[z_{ijk} \mapsto \alpha_i \beta_j \gamma_k\]
for $(i,j,k)$ in the set $Q = \{(1,0,0), (2,0,0), (0,0,1), (2,0,1), (0,1,0), (1,1,0), (0,1,1)\}$. This monomial map factors through the quasi-independence map $z_{ijk} \mapsto x_iy_{jk}$ associated to the following matrix:
\[
\begin{bNiceMatrix}[first-row,first-col] 
& y_{00} & y_{01} & y_{10} & y_{11} \\
x_0 & 0 & z_{001} & z_{010} & z_{011} \\
x_1 & z_{100} & 0 & z_{110} & 0 \\
x_2 & z_{200} & z_{201} & 0 & 0 
       \end{bNiceMatrix}
       \]
By Theorem~\ref{thm:gbofquasi-independence}, the kernel of this quasi-independence map is generated by the following cubic:
\[z_{001}z_{110}z_{200} - z_{010}z_{100}z_{201}.\]
%Another observation confirming Theorem~\ref{thm: quadratic and square-free condition for quasi-independence ideals} is the fact that $G_S$ is a $6$-cycle with an edge glued via a vertex. 
By Theorem \ref{thm:quasiIndepPseudoGrobner}, it remains to compute the lifts of the generators of the following two maps:
\[x_i \mapsto \alpha_i \quad y_{jk} \mapsto \beta_j\gamma_k.\]
The kernel of the former is the trivial ideal, so there are no lifts. Choose any weight order with weight $\omega_1$ on this trivial ideal. The kernel of the latter is generated by $y_{00}y_{11} - y_{01}y_{10}$. We have the following sets of lifts for each monomial:
\[\lift(y_{00}y_{11}) = \{\color{red}z_{100}z_{011}\color{black},\color{blue}z_{200}z_{011}\color{black}\} \quad \lift(y_{01}y_{10}) = \{z_{001}z_{010}, \color{red}z_{001}z_{110}\color{black},\color{blue} z_{201}z_{010}\color{black}, z_{201}z_{110}\}\]
Note that for each monomial in $\lift(y_{00}y_{11})$, the multiset of first coordinates also appears as a multiset of first coordinates for some monomial in $\lift(y_{01}y_{10})$ (this inclusion is depicted by a coloring of the monomials). Consequently, for any weight $\omega_2$ such that the weight order chooses leading term $y_{00}y_{11}$, we can complete all lifts of $y_{00}y_{11}$ to lifts of $y_{00}y_{11} - y_{01}y_{10}$. We obtain the following lifts:
\[z_{100}z_{011} - z_{001}z_{110} \quad \text{and} \quad z_{200}z_{011} - z_{201}z_{010}.\]
Together with the generator from the quasi-independence ideal, these lifts form a pseudo-Gr\"obner basis with respect to the weight order induced by $(\omega_1,\omega_2)^T B_Q$ where $Q$. In contrast to the toric fiber product, the quasi-independence gluing of two ideals which are quadratically generated can still yield an ideal which is not quadratically generated.
\end{example}

As we saw in Example \ref{example:weaklySHomogeneousLifting}, one must know the leading term of a polynomial to determine that it is weakly $Q$-homogeneous and construct the set of lifts. Consequently for this more general construction, one must keep track of the weight order through each iterative gluing. This differs from the toric fiber product, where one can obtain a Gr\"obner basis with respect to some weight order without ever computing the weight. This motivates the following definition.

\begin{definition}
Let $Q \subseteq [r] \times [s]$ and $f \subset \kk[x_1, \ldots, x_r]$ be a polynomial which is homogeneous with respect to total degree. Then $f$ is \emph{strongly $Q$-homogeneous} if for every two monomials $m = x_{j_1} x_{j_2} \ldots x_{j_d}$ and $m' = x_{j_1'} x_{j_2'} \ldots x_{j_d'}$ of $f$ there exists a 
permutation $\sigma$ of the elements of $[d]$ such that
% re-ordering of the $j_\ell'$ such that
\[
\{ 
(k_1, \ldots, k_d) \in [s]^d ~|~ (j_\ell, k_\ell) \in Q, 1 \leq \ell \leq d
\} 
= 
\{ 
(k_1, \ldots, k_d) \in [s]^d ~|~ (j'_{\sigma(\ell)}, k_\ell) \in Q, 1 \leq \ell \leq d
\}.
% \{ 
% (k_1, \ldots, k_d) \in [s]^d ~|~ (j'_\ell, k_\ell) \in Q, 1 \leq \ell \leq d \}. 
% \}
\]
\end{definition}

When $f$ is strongly $Q$-homogeneous, the subset relation in the definition of weakly $Q$-homogeneous holds regardless of which monomial is the leading term. Therefore it is unnecessary to keep track of the term order because it is not needed to compute the lifts of each generator. In this way, the strong version of $Q$-homogeneity is similar to the toric fiber product. 
In Section~\ref{section:CIMTree}, we use strong $Q$-homogeneity to compute a Gr\"obner basis for toric ideals associated to the face $\CIM_T$ of the characteristic imset polytope for $T$ a tree.

While the preceding results hold for ideals $I$ and $J$ which are not toric, there is a nice interpretation of the construction of $I \times_Q J$ when $I$ and $J$ are toric. 

\begin{corollary}
Let $I_A$ and $I_{A'}$ be toric ideals associated to integer matrices $A$ and $A'$. Then $I_A \times_Q I_{A'}$ is the toric ideal of the matrix $A \times_Q A'$ whose columns are all vectors of the form
$(a_j, a_k')^T$ such that $(j, k) \in Q$. 
\end{corollary}
\begin{proof}Let $\psi_A$, $\psi_{A'}$, and $\psi_{A \times_Q A'}$ be the monomial maps with associated integer matrices $A$, $A'$, and $A \times_Q A'$. The following diagram commutes
\begin{center}
  \begin{tikzcd}
\kk[z] \arrow[d, swap,"\phi_Q"] \arrow[r, "\psi_{A \times_Q A'}"] & \kk[t] \\
\kk[x, y] \arrow[ur,swap,"\psi_{A, A'}"]
\end{tikzcd}  
\end{center}
where
\begin{align*}
\psi_{A, A'} \colon \kk[x,y]&\rightarrow \kk[t]\\
x_{j} &\mapsto \psi_{A}(x_{j}) \\ y_{k} &\mapsto \psi_{A'}(y_{k}). 
\end{align*}
Hence we have the factorization $\psi(z_{jk}) = \psi_{A,A'}(x_j y_k) = \psi_A (x_j) \psi_{A'}(y_k)$, whenever $(j,k) \in Q$, which concludes the proof.
\end{proof}

\begin{example}
In Example~\ref{example:weaklySHomogeneousLifting}, we consider the quasi-independence gluing of the trivial ideal with the ideal for the image of the Segre embedding $\mathbb{P}^1 \times \mathbb{P}^1$ with respect to the given $Q$. Thus we obtain the following integer matrix: 
\[
A \times_Q A' = \begin{bNiceMatrix}[first-row,first-col]
& z_{001} & z_{010} & z_{011} & z_{100} & z_{110} & z_{200} & z_{201} \\
\alpha_0 & 1 & 1 & 1 & 0 & 0 & 0 & 0 \\
\alpha_1 & 0 & 0 & 0 & 1 & 1 & 0 & 0 \\
\alpha_2 & 0 & 0 & 0 & 0 & 0 & 1 & 1 \\
\beta_0  & 1 & 0 & 0 & 1 & 0 & 1 & 1 \\
\beta_1  & 0 & 1 & 1 & 0 & 1 & 0 & 0 \\
\gamma_0 & 0 & 1 & 0 & 1 & 1 & 1 & 0 \\
\gamma_1 & 1 & 0 & 1 & 0 & 0 & 0 & 1
\end{bNiceMatrix}.
\]
\end{example}

\section{Characteristic Imset Ideals via Quasi-Independence}
\label{section:CIMTree}

In this section we define the toric ideal of $\CIM_G$ and show that when $G$ is a tree, the ideal is an iterated quasi-independence gluing. We then use this to show that these ideals have square-free quadratic Gr\"obner bases. We use $\patp$ to denote a general pattern without reference to a specific DAG.

The vertices of $\cim_G$ correspond to the Markov equivalence classes with underlying skeleton $G$ \cite{studeny2010characteristic}. 
Consequently, we can equivalently represent the characteristic imset $c_\cg$ for $\cg$ with the pattern $\patp$ of the Markov equivalence class of $\cg$.
We do so in the following.
Given a pattern $\patp$, we let $\cg(\patp)$ denote any (fixed) DAG in the Markov equivalence class represented by $\patp$.
% Consequently we can define and compute $c_\patp$ for $\patp$ a pattern, instead of a DAG. That is, for a pattern $\patp$,
% \[
% c_\patp(S) = 
% \begin{cases}
% 1, ~\mbox{if}~ |S|=2 ~\mbox{and}~ S ~\mbox{is an edge of}~G \\
% 1, ~\mbox{if}~ |S|>2 ~\mbox{and there exists}~ i \in S ~\mbox{such that for all}~ j \in S \setminus \{i\}, ~j \rightarrow i~\mbox{in}~\patp \\
% 0, ~\mbox{otherwise}.
% \end{cases}
% \]

\begin{definition}
Let $G = ([p], E)$ be a graph and $\patp_1, \ldots, \patp_n$ be patterns which represent the Markov equivalence classes with skeleton $G$. Then the \emph{characteristic imset ideal} of $G$ is the toric ideal associated to $\cim_G$, which is the kernel of the map
\begin{align*}
\psi_G : \kk[z_{\patp_1}, \ldots, z_{\patp_n}]& \to \kk[t_S ~|~ S \subseteq [p], |S| \geq 2] \\
z_{\patp_i} &\mapsto t^{c_{\cg(\patp_i)}}. 
\end{align*}
We denote this ideal by $I_G = \ker(\psi_G)$. 
\end{definition}

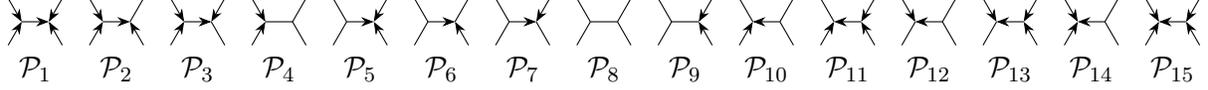
\begin{figure}[h]
\centering
\begin{tikzpicture}[scale = 0.36]

\begin{scope}[shift={(3,0)}]
\draw [Stealth-] (1,0)--(0,0);
\draw [rotate = 120, -Stealth] (1,0)--(0,0);
\draw [rotate = 240, -Stealth] (1,0)--(0,0);
\draw [xshift=1cm, rotate = 60, -Stealth] (1,0)--(0,0);
\draw [xshift=1cm, rotate = -60, -Stealth] (1,0)--(0,0);
\node at (0.5,-1.75) {$\patp_1$};
\end{scope}

\begin{scope}[shift={(6,0)}]
\draw [Stealth-] (1,0)--(0,0);
\draw [rotate = 120, -Stealth] (1,0)--(0,0);
\draw [rotate = 240, -Stealth] (1,0)--(0,0);
\draw [xshift=1cm, rotate = 60] (1,0)--(0,0);
\draw [xshift=1cm, rotate = -60, -Stealth] (1,0)--(0,0);
\node at (0.5,-1.75) {$\patp_2$};
\end{scope}

\begin{scope}[shift={(9,0)}]
\draw [Stealth-] (1,0)--(0,0);
\draw [rotate = 120, -Stealth] (1,0)--(0,0);
\draw [rotate = 240, -Stealth] (1,0)--(0,0);
\draw [xshift=1cm, rotate = 60, -Stealth] (1,0)--(0,0);
\draw [xshift=1cm, rotate = -60] (1,0)--(0,0);
\node at (0.5,-1.75) {$\patp_3$};
\end{scope}

\begin{scope}[shift={(12,0)}]
\draw (1,0)--(0,0);
\draw [rotate = 120, -Stealth] (1,0)--(0,0);
\draw [rotate = 240, -Stealth] (1,0)--(0,0);
\draw [xshift=1cm, rotate = 60] (1,0)--(0,0);
\draw [xshift=1cm, rotate = -60] (1,0)--(0,0);
\node at (0.5,-1.75) {$\patp_4$};
\end{scope}

\begin{scope}[shift={(15,0)}]
\draw [Stealth-] (1,0)--(0,0);
\draw [rotate = 120] (1,0)--(0,0);
\draw [rotate = 240] (1,0)--(0,0);
\draw [xshift=1cm, rotate = 60, -Stealth] (1,0)--(0,0);
\draw [xshift=1cm, rotate = -60, -Stealth] (1,0)--(0,0);
\node at (0.5,-1.75) {$\patp_5$};
\end{scope}

\begin{scope}[shift={(18,0)}]
\draw [Stealth-] (1,0)--(0,0);
\draw [rotate = 120] (1,0)--(0,0);
\draw [rotate = 240] (1,0)--(0,0);
\draw [xshift=1cm, rotate = 60] (1,0)--(0,0);
\draw [xshift=1cm, rotate = -60, -Stealth] (1,0)--(0,0);
\node at (0.5,-1.75) {$\patp_6$};
\end{scope}

\begin{scope}[shift={(21,0)}]
\draw [Stealth-] (1,0)--(0,0);
\draw [rotate = 120] (1,0)--(0,0);
\draw [rotate = 240] (1,0)--(0,0);
\draw [xshift=1cm, rotate = 60, -Stealth] (1,0)--(0,0);
\draw [xshift=1cm, rotate = -60] (1,0)--(0,0);
\node at (0.5,-1.75) {$\patp_7$};
\end{scope}

\begin{scope}[shift={(24,0)}]
\draw (1,0)--(0,0);
\draw [rotate = 120] (1,0)--(0,0);
\draw [rotate = 240] (1,0)--(0,0);
\draw [xshift=1cm, rotate = 60] (1,0)--(0,0);
\draw [xshift=1cm, rotate = -60] (1,0)--(0,0);
\node at (0.5,-1.75) {$\patp_8$};
\end{scope}

\begin{scope}[shift={(27,0)}]
\draw (1,0)--(0,0);
\draw [rotate = 120] (1,0)--(0,0);
\draw [rotate = 240] (1,0)--(0,0);
\draw [xshift=1cm, rotate = 60, -Stealth] (1,0)--(0,0);
\draw [xshift=1cm, rotate = -60, -Stealth] (1,0)--(0,0);
\node at (0.5,-1.75) {$\patp_9$};
\end{scope}

\begin{scope}[shift={(30,0)}]
\draw [-Stealth] (1,0)--(0,0);
\draw [rotate = 120] (1,0)--(0,0);
\draw [rotate = 240, -Stealth] (1,0)--(0,0);
\draw [xshift=1cm, rotate = 60] (1,0)--(0,0);
\draw [xshift=1cm, rotate = -60] (1,0)--(0,0);
\node at (0.5,-1.75) {$\patp_{10}$};
\end{scope}

\begin{scope}[shift={(33,0)}]
\draw [-Stealth] (1,0)--(0,0);
\draw [rotate = 120] (1,0)--(0,0);
\draw [rotate = 240, -Stealth] (1,0)--(0,0);
\draw [xshift=1cm, rotate = 60, -Stealth] (1,0)--(0,0);
\draw [xshift=1cm, rotate = -60, -Stealth] (1,0)--(0,0);
\node at (0.5,-1.75) {$\patp_{11}$};
\end{scope}

\begin{scope}[shift={(36,0)}]
\draw [-Stealth] (1,0)--(0,0);
\draw [rotate = 120, -Stealth] (1,0)--(0,0);
\draw [rotate = 240] (1,0)--(0,0);
\draw [xshift=1cm, rotate = 60] (1,0)--(0,0);
\draw [xshift=1cm, rotate = -60] (1,0)--(0,0);
\node at (0.5,-1.75) {$\patp_{12}$};
\end{scope}

\begin{scope}[shift={(39,0)}]
\draw [-Stealth] (1,0)--(0,0);
\draw [rotate = 120, -Stealth] (1,0)--(0,0);
\draw [rotate = 240] (1,0)--(0,0);
\draw [xshift=1cm, rotate = 60, -Stealth] (1,0)--(0,0);
\draw [xshift=1cm, rotate = -60, -Stealth] (1,0)--(0,0);
\node at (0.5,-1.75) {$\patp_{13}$};
\end{scope}

\begin{scope}[shift={(42,0)}]
\draw [-Stealth] (1,0)--(0,0);
\draw [rotate = 120, -Stealth] (1,0)--(0,0);
\draw [rotate = 240, -Stealth] (1,0)--(0,0);
\draw [xshift=1cm, rotate = 60] (1,0)--(0,0);
\draw [xshift=1cm, rotate = -60] (1,0)--(0,0);
\node at (0.5,-1.75) {$\patp_{14}$};
\end{scope}

\begin{scope}[shift={(45,0)}]
\draw [-Stealth] (1,0)--(0,0);
\draw [rotate = 120, -Stealth] (1,0)--(0,0);
\draw [rotate = 240, -Stealth] (1,0)--(0,0);
\draw [xshift=1cm, rotate = 60, -Stealth] (1,0)--(0,0);
\draw [xshift=1cm, rotate = -60, -Stealth] (1,0)--(0,0);
\node at (0.5,-1.75) {$\patp_{15}$};
\end{scope}
\end{tikzpicture}
\caption{A list of all patterns on the quartet tree.}
\label{fig:quartetGB}
\end{figure}

\begin{example}[Quartet Tree]
\label{ex:quartet}
Let $G$ be the skeleton of the patterns appearing in Figure~\ref{fig:quartetGB}, called the \emph{quartet}. If each pattern with skeleton $G$ is labeled as in Figure~\ref{fig:quartetGB}, then $I_G$ has the following generators:

\[\begin{array}{cccc}
&z_{\patp_1}z_{\patp_6} - z_{\patp_2}z_{\patp_5},         &z_{\patp_1}z_{\patp_7} - z_{\patp_3}z_{\patp_5}, &z_{\patp_1}z_{\patp_8} - z_{\patp_4}z_{\patp_5},\\
&z_{\patp_2}z_{\patp_7} - z_{\patp_3}z_{\patp_6}, &z_{\patp_2}z_{\patp_8} - z_{\patp_4}z_{\patp_6}, &z_{\patp_3}z_{\patp_8} - z_{\patp_4}z_{\patp_7},\\
&z_{\patp_8}z_{\patp_{11}} - z_{\patp_9}z_{\patp_{10}}, &z_{\patp_8}z_{\patp_{13}} - z_{\patp_9}z_{\patp_{12}}, &z_{\patp_8}z_{\patp_{15}} - z_{\patp_9}z_{\patp_{14}},\\
&z_{\patp_{10}}z_{\patp_{13}} - z_{\patp_{11}}z_{\patp_{12}}, &z_{\patp_{10}}z_{\patp_{15}} - z_{\patp_{11}}z_{\patp_{14}}, &z_{\patp_{12}}z_{\patp_{15}} -z_{\patp_{13}}z_{\patp_{14}}.
\end{array}\]
\end{example}

We now focus on the case where the undirected graph is a tree $T = ([p], E)$. If $e = u - v$ is a non-leaf edge $T$, then deleting $e$ from $T$ yields a graph with two connected components. 
% \textcolor{red}{L: we sometimes use $(u, v)$ for a directed edge and sometimes $u\rightarrow v$.  The former is classic graph theory but the latter is more obvious in my opinion.  I tend to use the latter.  In any case we should pick one notation and stick with it.}
% \textcolor{joegreen}{I think I fixed this issue.} {\color{orange} I checked this for undirected edges as well, although the $u - v$ notation seems a little odd.}
Let $T_u'$ and $T_v'$ be the connected components containing the vertices $u$ and $v$ respectively. Denote by $T_u$ the tree obtained by adding the edge $e$ back into $T_u'$, that is $E(T_u) = E(T_u') \cup \{e\}$ and define $T_v$ analogously. 

\begin{definition}
\label{defn:parting}
Let $T = ([p], E)$ be an undirected tree and let $\ct \in \meq(T)$ be a pattern representing a Markov equivalence class of $T$. Let $e = u - v \in T$ be a non-leaf edge.
Then the \emph{parting} of $\ct$ at $e$ is the pair of patterns $\parting(\ct, e) = (\ct_u, \ct_v)$ such that $\ct_u$ and $\ct_v$ are the induced directed subgraphs of $\ct$ on vertices of $T_u$ and $T_v$ respectively, with the orientation of $e$ in $\ct_u$ and $\ct_v$ is modified by:
\begin{enumerate}
\item If $u \to v \in \ct$ then $u \to v \in \ct_v$ and $u - v \in \ct_u$
\item If $v \to u \in \ct$ then $v \to u \in \ct_u$ and $u - v \in \ct_v$
\item If $u - v \in \ct$ then $u - v \in \ct_u$ and $u - v \in \ct_v$
\end{enumerate}
\end{definition}
% \begin{definition}
% Let $T = ([p], E)$ be an undirected tree and let $\ct \in \meq(T)$ be a pattern representing a Markov equivalence class of $T$. Let $e = u - v \in T$ be non-leaf edge.
% Then the \emph{parting} of $\ct$ at $e$ is the pair of patterns $\parting(\ct, e) = (\ct_u, \ct_v)$ such that the skeleton of $\ct_i$ is $T_i$ for $i = u, v$ and the orientation of $e$ in $\ct_u$ and $\ct_v$ is given by:
% \begin{enumerate}
% \item If $u \to v \in \ct$ then $u \to v \in \ct_v$ and $u - v \in \ct_u$
% \item If $v \to u \in \ct$ then $v \to u \in \ct_u$ and $u - v \in \ct_v$
% \item If $u - v \in \ct$ then $u - v \in \ct_u$ and $u - v \in \ct_v$
% \end{enumerate}
% \end{definition}
In words, we take the induced partially directed acyclic graphs on $T_u$ and $T_v$, removing the orientation on $u-v$ if it does not appear in a v-structure in the subgraph.
%\textcolor{red}{In words, the parting breaks the tree into two subtrees corresponding to the two connected components resulting from removing the edge, but where we include the separating edge in both subtrees. If the edge is oriented in the pattern, it must be in a v-structure, either with center node $u$ or center node $v$, but it cannot be in both.  Hence, the v-structure appears in which ever of the subtrees it has both edges in, and the edge becomes unoriented in the other tree (as it is not in a v-structure in that subtree.  These are cases (1) and (2) in the above definition.  
%If the edge is not oriented in the pattern $\ct$ then it should not appear in a v-structure in either subtree, which is case (3). }
This operation is depicted in Figure~\ref{fig:parting}. 

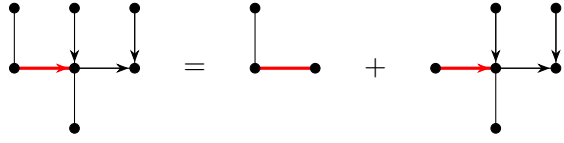
\begin{figure}
\centering
\begin{tikzpicture}[scale = 0.8]

\draw (-1,2)--(-1,1)--(1,1)--(1,2);
\draw (0,0)--(0,2);

\draw [-Stealth] (0,2)--(0,1.1);
\draw [-Stealth] (0,1)--(0.9,1);
\draw [-Stealth] (1,2)--(1,1.1);

\draw [red, very thick] (-1,1)--(0,1);
\draw [red,-Stealth] (-1,1)--(-0.1,1);

\draw [fill=black] (0,0) circle [radius=0.08cm];

\foreach \x in {-1,...,1}{
\foreach \y in {1,...,2}{
\draw [fill=black] (\x,\y) circle [radius=0.08cm];
}
}

\node at (2,1) {=};

\begin{scope}[shift={(4,0)}]
\draw (-1,2)--(-1,1)--(0,1);
\draw [red, very thick] (-1,1)--(0,1);

\draw [fill=black] (-1,2) circle [radius=0.08cm];
\draw [fill=black] (-1,1) circle [radius=0.08cm];
\draw [fill=black] (0,1) circle [radius=0.08cm];
\end{scope}

\node at (5,1) {+};

\begin{scope}[shift={(7,0)}]
\draw (-1,1)--(1,1)--(1,2);
\draw (0,0)--(0,2);

\draw [-Stealth] (0,2)--(0,1.1);
\draw [-Stealth] (0,1)--(0.9,1);
\draw [-Stealth] (1,2)--(1,1.1);

\draw [red,-Stealth] (-1,1)--(-0.1,1);
\draw [red, very thick] (-1,1)--(0,1);

\draw [fill=black] (0,2) circle [radius=0.08cm];
\draw [fill=black] (1,2) circle [radius=0.08cm];

\draw [fill=black] (-1,1) circle [radius=0.08cm];
\draw [fill=black] (0,1) circle [radius=0.08cm];
\draw [fill=black] (1,1) circle [radius=0.08cm];

\draw [fill=black] (0,0) circle [radius=0.08cm];
\end{scope}
\end{tikzpicture}
\caption{The graph on the left parts along the red edge.}
\label{fig:parting}
\end{figure}

\begin{theorem}
\label{thm:cimTreeQIG}
Let $T = ([p], E)$ be a tree, $e = u - v$ be a non-leaf edge of $T$.% and $T_u, T_v$ be the undirected trees obtained from parting $T$ at $e$. Let $I_T, I_{T_u}, I_{T_v}$ be the CIM toric ideals of $T, T_u, T_v$ respectively. 
Then $I_T = I_{T_u} \times_Q I_{T_v}$ where $Q$ is the set $Q = \{\parting(\ct, e) ~|~ \ct \in \meq(T)\}$. 
\end{theorem}
\begin{proof}
To prove this we extend the main technique introduced in \cite[Section 3]{sullivant2007toric} to show that ideals are toric fiber products to QIGs. 

Let $\ct \in \meq(T)$ and $\parting(\ct, e) = (\ct_u, \ct_v)$. Let the ambient rings of $I_T, I_{T_u}$ and $I_{T_v}$ be  $\kk[z] = \kk[z_\ct ~|~ \ct \in \meq(T)]$, $\kk[x] = \kk[x_{\ct_u} ~|~ \ct_u \in \meq(T_u)]$, and $\kk[y] = \kk[y_{\ct_v} ~|~ \ct_v \in \meq(T_v)]$ respectively. If we square the parameter 
$t_{\{u, v \}}$ everywhere it appears in $\psi_T$ then
\begin{equation}
\label{eqn:treeMapFactors}
\psi_T(z_\ct) = \psi_{T_u}(x_{\ct_u}) \psi_{T_v}(y_{\ct_v}).   
\end{equation}

This follows because if $S \subseteq [p]$ is the vertex set of a star subgraph of $T$ (so that $t_S$ appears in $\psi_T(z_\ct)$ or $\psi_{T_u}(x_{\ct_u})\psi_{T_v}(y_{\ct_v})$) and $S \neq \{u, v\}$, then $S$ is a subset of the vertices of $T_u$ or $T_v$ and thus either $T|_S = T_u |_S$ or $T|_S = T_v |_S$. So then the parameter $t_S$ appears in $\psi_T(z_\ct)$ if and only if $t_S$ it appears in either $\psi_{T_u}(x_{\ct_u})$ or $\psi_{T_v}(y_{\ct_v})$. Note that squaring the parameter $t_{\{u, v \}}$ everywhere it appears in $\psi_T$ does not change the kernel of the map whatsoever. The factorization in Equation~\ref{eqn:treeMapFactors} implies that the following diagram commutes

\begin{center}
  \begin{tikzcd}
\kk[z] \arrow[d, "\phi_Q"] \arrow[r, "\psi_T"] & \kk[t_S ~|~ S \subseteq [p], |S| \geq 2] \\
\kk[x, y] \arrow[ur, "\psi_{T_u, T_v}"]
\end{tikzcd}  
\end{center}
where the map $\psi_{T_u, T_v}$ is given by
\begin{align*}
\psi_{T_u, T_v}: \kk[x,y] &\to \kk[t_S] \\
x_{\ct_u} &\mapsto \psi_{T_u}(x_{\ct_u}) \\ 
y_{\ct_v} &\mapsto \psi_{T_v}(y_{\ct_v}). 
\end{align*}
But recall that $I_T = \ker(\psi_T) = \psi_T^{-1}(0)$. Since the above diagram commutes, we have
\[
\psi_T^{-1}(0) = \phi_Q^{-1}(\psi_{T_u, T_v}^{-1}(0)) = \phi_Q^{-1}(I_{T_u} + I_{T_v}) = I_{T_u} \times_Q I_{T_v} 
\]
which completes the proof.
\end{proof}

\begin{example}[Quartet Tree] In general, to compute a Gr\"obner basis for a QIG one must show that the defining monomial map factors through a quasi-independence map $\phi_Q$ and that the two ideals that we wish to glue together are weakly $Q$-homogeneous. Let $T$ be the quartet graph, with vertices labeled by elements of $[6]$ and edges $\{ 1 - 3, 2 - 3, 3 - 4, 4 - 5, 4 - 6\}$. Let $Q \subseteq \meq(T|_{1234}) \times \meq(T|_{3456})$ be the set of all partings of patterns with skeleton $T$. %pairs of patterns such that the first coordinate has no v-structure with sink vertex 3 and the second coordinate has no v-structure with sink vertex 4.
%the edge $3 - 4$ does not have orientation $3 \rightarrow 4$ in one pattern and orientation $3 \leftarrow 4$ in the other.

%\textcolor{joegreen}{\textbf{NEW COMMENT:} I think I got rid of the essential graph everywhere in sections 4 and 5.} {\color{orange}PDAGs of two star graphs with vertices 1,2,3,4 and 3,4,5,6, or so?}
%\textcolor{red}{L: I see what it happening now and the distinction is very important.  We are using compatible essential graph structures to define the structural zeros, but once we do that we can reduce our computations to considering only the patterns for these essential graphs.  For instance, this is why the entry $(1,5)$ in the matrix in Figure 3 must be a structural zero.  Regarding my previous comment: one could take the toric fiber product where we glue everything.  Some gluings won't result in characteristic imsets these are the vertices of the polytope we would disregard.  The question is whether or not all binomials in the GB for this toric fiber product that use variables indexed by characteristic imsets using ONLY variables indexed by characteristic imsets.  In this case we could recover our GB via elimination.}

We claim $\psi_T$ factors. Consider the pattern $\ct$ with directed edges $1 \rightarrow 3$, $2 \rightarrow 3$, $3 \rightarrow 4$, and $6 \rightarrow 4$. Then under $\psi_T$ we have
\[z_\ct \mapsto t_{13}t_{23}t_{123}t_{34}t_{45}t_{46}t_{346}.\]
After squaring $t_{34}$ (which does not change the kernel), we have the following factorization:
\[z_{\ct} \mapsto x_{\ct_3}y_{\ct_4}, \quad x_{\ct_3} \mapsto t_{13}t_{23}t_{34}t_{123} \quad y_{\ct_4} \mapsto t_{34}t_{45}t_{46}t_{346}\]
where $\ct_3$ and $\ct_4$ are defined as in Definition~\ref{defn:parting}. Similarly we see this factorization for all other patterns.

For the quartet (and more generally for all bistars) the two trees we glue together are star graphs. Consequently their characteristic imset ideals are trivial and are strongly $Q$-homogeneous. It follows that the QIG is $\phi_Q^{-1}(0)$, in particular it is the quasi-independence ideal of the bipartite graph $G_Q$. By Theorem~\ref{thm: quadratic and square-free condition for quasi-independence ideals}, the kernel of $\phi_Q$ is given by the $2 \times 2$ minors of the matrix in Figure~\ref{fig:bistarGluingRule} that contain no structural zeros, which are the binomials appearing in Example~\ref{ex:quartet}. 

Recall that toric fiber products are QIGs associated to block diagonal matrices. In Figure~\ref{fig:bistarGluingRule} we have two overlapping blocks, one for each possible orientation of the gluing edge. These blocks overlap in the pattern with no v-structures, and so this characteristic imset ideal fails to be a toric fiber product. 

%\textcolor{joegreen}{Is there a simple algebro-geometric reason why this is true, even up to (not necessarily linear) change of coordinate? Note that for trees, the dimension of variety after QIG is the sum of the dimensions of the varieties that we glue. Can we use this?} {\color{orange} I guess sum of the dimensions minus one: for this example 5+5-1. I think if these were block diagonal matrices, to calculate the dimension we would sum up the number of rows and columns and then substract the number of blocks.}

%Recall that toric fiber products are QIGs associated to block diagonal matrices. In Figure~\ref{fig:bistarGluingRule} we have two blocks, one where the gluing edge is allowed to be oriented leftward (in blue) and another where it can be oriented rightward (in red). These blocks overlap in the pattern with no v-structures. If one were to attempt to represent this gluing as a toric fiber product, the patterns with no v-structures in the index sets of the rows and columns would need to be assigned two multigradings so that they could glue to patterns with either orientation of the gluing edge. Characteristic imset ideals associated to trees fail to be toric fiber products because the candidate multigradings that define the gluing rule fail to be functions.

\begin{figure}[h]
\centering
\begin{tikzpicture}[scale=0.33]
%BIG BISTAR
\begin{scope}[shift = {(-8,-5)}, scale = 3]
\draw (1,0)--(0,0);
\draw [rotate = 120] (1,0)--(0,0);
\draw [rotate = 240] (1,0)--(0,0);
\draw [xshift=1cm, rotate = 60] (1,0)--(0,0);
\draw [xshift=1cm, rotate = -60] (1,0)--(0,0);
\node at (0.5,-1.3) {$T$};

\node at (-0.7,0.866) {1};
\node at (-0.7,-0.866) {2};
\node at (-0.3,0) {3};
\node at (1.3,0) {4};
\node at (1.7,0.866) {5};
\node at (1.7,-0.866) {6};
\end{scope}

%EVERYTHING ELSE
\draw [thick] (2,-11.1)--(1.6,-11.1)--(1.6,1.1)--(2,1.1);
\draw [thick] (17,-11.1)--(17.4,-11.1)--(17.4,1.1)--(17,1.1);

%ROW INDICES
\begin{scope}
\draw (1,0)--(0,0);
\draw [rotate = 120, -Stealth] (1,0)--(0,0);
\draw [rotate = 240, -Stealth] (1,0)--(0,0);
\end{scope}

\begin{scope}[shift={(0,-2.5)}]
\draw (1,0)--(0,0);
\draw [rotate = 120] (1,0)--(0,0);
\draw [rotate = 240] (1,0)--(0,0);
\end{scope}

\begin{scope}[shift={(0,-5)}]
\draw [-Stealth] (1,0)--(0,0);
\draw [rotate = 120] (1,0)--(0,0);
\draw [rotate = 240, -Stealth] (1,0)--(0,0);
\end{scope}

\begin{scope}[shift={(0,-7.5)}]
\draw [-Stealth] (1,0)--(0,0);
\draw [rotate = 120, -Stealth] (1,0)--(0,0);
\draw [rotate = 240] (1,0)--(0,0);
\end{scope}

\begin{scope}[shift={(0,-10)}]
\draw [-Stealth] (1,0)--(0,0);
\draw [rotate = 120, -Stealth] (1,0)--(0,0);
\draw [rotate = 240, -Stealth] (1,0)--(0,0);
\end{scope}

%COLUMN INDICES
\begin{scope}[shift={(2.75,2.5)}]
\draw [Stealth-] (1,0)--(0,0);
\draw [xshift=1cm, rotate = 60, -Stealth] (1,0)--(0,0);
\draw [xshift=1cm, rotate = -60, -Stealth] (1,0)--(0,0);
\end{scope}

\begin{scope}[shift={(5.75,2.5)}]
\draw [Stealth-] (1,0)--(0,0);
\draw [xshift=1cm, rotate = 60] (1,0)--(0,0);
\draw [xshift=1cm, rotate = -60, -Stealth] (1,0)--(0,0);
\end{scope}

\begin{scope}[shift={(8.75,2.5)}]
\draw [Stealth-] (1,0)--(0,0);
\draw [xshift=1cm, rotate = 60, -Stealth] (1,0)--(0,0);
\draw [xshift=1cm, rotate = -60] (1,0)--(0,0);
\end{scope}

\begin{scope}[shift={(11.75,2.5)}]
\draw (1,0)--(0,0);
\draw [xshift=1cm, rotate = 60] (1,0)--(0,0);
\draw [xshift=1cm, rotate = -60] (1,0)--(0,0);
\end{scope}

\begin{scope}[shift={(14.75,2.5)}]
\draw (1,0)--(0,0);
\draw [xshift=1cm, rotate = 60, -Stealth] (1,0)--(0,0);
\draw [xshift=1cm, rotate = -60, -Stealth] (1,0)--(0,0);
\end{scope}

\draw [red, fill = red, opacity = 0.25] (2.3,1.2)--(13.7,1.2)--(13.7,-3.7)--(2.3,-3.7);
\draw [blue, fill = blue, opacity = 0.25] (11.3,-1.3)--(16.7,-1.3)--(16.7,-11.2)--(11.3,-11.2);

%ROW 1
\begin{scope}[shift={(3,0)}]
\draw [Stealth-] (1,0)--(0,0);
\draw [rotate = 120, -Stealth] (1,0)--(0,0);
\draw [rotate = 240, -Stealth] (1,0)--(0,0);
\draw [xshift=1cm, rotate = 60, -Stealth] (1,0)--(0,0);
\draw [xshift=1cm, rotate = -60, -Stealth] (1,0)--(0,0);
\end{scope}

\begin{scope}[shift={(6,0)}]
\draw [Stealth-] (1,0)--(0,0);
\draw [rotate = 120, -Stealth] (1,0)--(0,0);
\draw [rotate = 240, -Stealth] (1,0)--(0,0);
\draw [xshift=1cm, rotate = 60] (1,0)--(0,0);
\draw [xshift=1cm, rotate = -60, -Stealth] (1,0)--(0,0);
\end{scope}

\begin{scope}[shift={(9,0)}]
\draw [Stealth-] (1,0)--(0,0);
\draw [rotate = 120, -Stealth] (1,0)--(0,0);
\draw [rotate = 240, -Stealth] (1,0)--(0,0);
\draw [xshift=1cm, rotate = 60, -Stealth] (1,0)--(0,0);
\draw [xshift=1cm, rotate = -60] (1,0)--(0,0);
\end{scope}

\begin{scope}[shift={(12,0)}]
\draw (1,0)--(0,0);
\draw [rotate = 120, -Stealth] (1,0)--(0,0);
\draw [rotate = 240, -Stealth] (1,0)--(0,0);
\draw [xshift=1cm, rotate = 60] (1,0)--(0,0);
\draw [xshift=1cm, rotate = -60] (1,0)--(0,0);
\end{scope}

\begin{scope}[shift={(15,0)}]
\node at (0.5,0) {0};
\end{scope}

%ROW 2
\begin{scope}[shift={(3,-2.5)}]
\draw [Stealth-] (1,0)--(0,0);
\draw [rotate = 120] (1,0)--(0,0);
\draw [rotate = 240] (1,0)--(0,0);
\draw [xshift=1cm, rotate = 60, -Stealth] (1,0)--(0,0);
\draw [xshift=1cm, rotate = -60, -Stealth] (1,0)--(0,0);
\end{scope}

\begin{scope}[shift={(6,-2.5)}]
\draw [Stealth-] (1,0)--(0,0);
\draw [rotate = 120] (1,0)--(0,0);
\draw [rotate = 240] (1,0)--(0,0);
\draw [xshift=1cm, rotate = 60] (1,0)--(0,0);
\draw [xshift=1cm, rotate = -60, -Stealth] (1,0)--(0,0);
\end{scope}

\begin{scope}[shift={(9,-2.5)}]
\draw [Stealth-] (1,0)--(0,0);
\draw [rotate = 120] (1,0)--(0,0);
\draw [rotate = 240] (1,0)--(0,0);
\draw [xshift=1cm, rotate = 60, -Stealth] (1,0)--(0,0);
\draw [xshift=1cm, rotate = -60] (1,0)--(0,0);
\end{scope}

\begin{scope}[shift={(12,-2.5)}]
\draw (1,0)--(0,0);
\draw [rotate = 120] (1,0)--(0,0);
\draw [rotate = 240] (1,0)--(0,0);
\draw [xshift=1cm, rotate = 60] (1,0)--(0,0);
\draw [xshift=1cm, rotate = -60] (1,0)--(0,0);
\end{scope}

\begin{scope}[shift={(15,-2.5)}]
\draw (1,0)--(0,0);
\draw [rotate = 120] (1,0)--(0,0);
\draw [rotate = 240] (1,0)--(0,0);
\draw [xshift=1cm, rotate = 60, -Stealth] (1,0)--(0,0);
\draw [xshift=1cm, rotate = -60, -Stealth] (1,0)--(0,0);
\end{scope}

%ROW 3
\begin{scope}[shift={(3,-5)}]
\node at (0.5,0) {0};
\end{scope}

\begin{scope}[shift={(6,-5)}]
\node at (0.5,0) {0};
\end{scope}

\begin{scope}[shift={(9,-5)}]
\node at (0.5,0) {0};
\end{scope}

\begin{scope}[shift={(12,-5)}]
\draw [-Stealth] (1,0)--(0,0);
\draw [rotate = 120] (1,0)--(0,0);
\draw [rotate = 240, -Stealth] (1,0)--(0,0);
\draw [xshift=1cm, rotate = 60] (1,0)--(0,0);
\draw [xshift=1cm, rotate = -60] (1,0)--(0,0);
\end{scope}

\begin{scope}[shift={(15,-5)}]
\draw [-Stealth] (1,0)--(0,0);
\draw [rotate = 120] (1,0)--(0,0);
\draw [rotate = 240, -Stealth] (1,0)--(0,0);
\draw [xshift=1cm, rotate = 60, -Stealth] (1,0)--(0,0);
\draw [xshift=1cm, rotate = -60, -Stealth] (1,0)--(0,0);
\end{scope}

%ROW 4
\begin{scope}[shift={(3,-7.5)}]
\node at (0.5,0) {0};
\end{scope}

\begin{scope}[shift={(6,-7.5)}]
\node at (0.5,0) {0};
\end{scope}

\begin{scope}[shift={(9,-7.5)}]
\node at (0.5,0) {0};
\end{scope}

\begin{scope}[shift={(12,-7.5)}]
\draw [-Stealth] (1,0)--(0,0);
\draw [rotate = 120, -Stealth] (1,0)--(0,0);
\draw [rotate = 240] (1,0)--(0,0);
\draw [xshift=1cm, rotate = 60] (1,0)--(0,0);
\draw [xshift=1cm, rotate = -60] (1,0)--(0,0);
\end{scope}

\begin{scope}[shift={(15,-7.5)}]
\draw [-Stealth] (1,0)--(0,0);
\draw [rotate = 120, -Stealth] (1,0)--(0,0);
\draw [rotate = 240] (1,0)--(0,0);
\draw [xshift=1cm, rotate = 60, -Stealth] (1,0)--(0,0);
\draw [xshift=1cm, rotate = -60, -Stealth] (1,0)--(0,0);
\end{scope}

%ROW 5
\begin{scope}[shift={(3,-10)}]
\node at (0.5,0) {0};
\end{scope}

\begin{scope}[shift={(6,-10)}]
\node at (0.5,0) {0};
\end{scope}

\begin{scope}[shift={(9,-10)}]
\node at (0.5,0) {0};
\end{scope}

\begin{scope}[shift={(12,-10)}]
\draw [-Stealth] (1,0)--(0,0);
\draw [rotate = 120, -Stealth] (1,0)--(0,0);
\draw [rotate = 240, -Stealth] (1,0)--(0,0);
\draw [xshift=1cm, rotate = 60] (1,0)--(0,0);
\draw [xshift=1cm, rotate = -60] (1,0)--(0,0);
\end{scope}

\begin{scope}[shift={(15,-10)}]
\draw [-Stealth] (1,0)--(0,0);
\draw [rotate = 120, -Stealth] (1,0)--(0,0);
\draw [rotate = 240, -Stealth] (1,0)--(0,0);
\draw [xshift=1cm, rotate = 60, -Stealth] (1,0)--(0,0);
\draw [xshift=1cm, rotate = -60, -Stealth] (1,0)--(0,0);
\end{scope}
\end{tikzpicture}
%\caption{The matrix associated to the QIG defining the characteristic imset polytope of the bistar. Structural zeros correspond to row and column indexing patterns where the gluing edge in every pair of associated DAGs has opposing orientations.}
\caption{The matrix of patterns associated to the QIG defining the characteristic imset polytope of the bistar. Structural zeros correspond to pairs of essential graphs indexing rows and columns where the gluing edge in both essential graphs has forced opposing orientations.}
\label{fig:bistarGluingRule}
\end{figure}
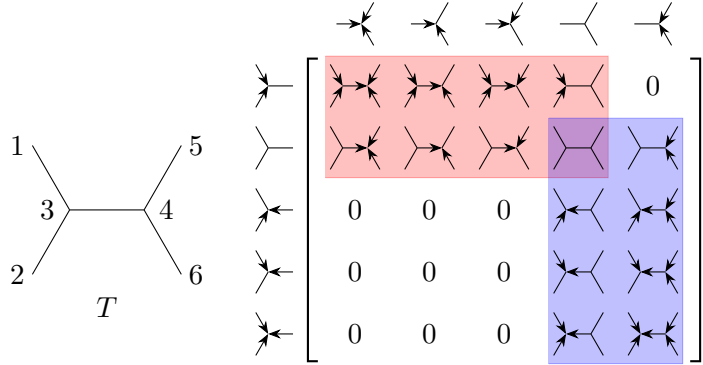

\end{example}

We have shown that the ideal $I_T = I_{T_u} \times_Q I_{T_v}$, but to use Theorem \ref{thm:quasiIndepPseudoGrobner} to find a Gr\"obner basis for $I_T$ we will show that $I_{T_u}$ and $I_{T_v}$ are actually strongly $Q$-homogeneous.   

\begin{lemma}
\label{lemma:cimTreeIdealHom}
Let $T = ([p], E)$ be a tree, $e = u - v$ be a non-leaf edge of $T$ and $T_u, T_v$ be the undirected trees obtained from parting $T$ at $e$. Let $Q = \{\parting(\ct, e) ~|~ \ct \in \meq(T)\}$. Then $I_{T_u}$ and $I_{T_v}$ are strongly $Q$-homogeneous. 
\end{lemma}
\begin{proof}
We will show that $I_{T_u}$ is strongly $Q$-homogeneous since the proof is the same for $I_{T_v}$. For each pattern with skeleton $T_u$, we have one of the following 3 orientations of $e$:
\[u \to v, \quad v \to u, \quad \text{or} \quad u - v.\]
To order the rows and columns of $A_Q$, we consider the essential graph associated to each pattern. Order the rows and columns such that all essential graphs with $e$ oriented $u \to v$ are first, then all essential graphs with $u - v$, and lastly the essential graphs with $v \to u$. The nonzero entries of this matrix are row and column convex. Therefore $G_Q$ is chordal bipartite and all generators from the quasi-independence map are quadratic. Since lifting does not change the degree of a generator, we may assume by induction that $I_{T_u}$ is quadratically generated. Furthermore, the coefficients of any generator are $\pm 1$ since all generators are liftings of $2 \times 2$ determinants. Suppose $f = x_{\patp_1} x_{\patp_2} - x_{\patp_3} x_{\patp_4}$ is a generator of $I_{T_u}$ and without loss of generality, let $(\patp_1,\patp_1'),(\patp_2,\patp_2') \in Q$. The multiset of forced orientations $v \to u$ in each monomial of $f$ must be the same since $f \in I_{T_u}$. Consequently we must have either $(\patp_3,\patp_1'),(\patp_4,\patp_2') \in Q$ or $(\patp_3,\patp_2'),(\patp_4,\patp_1') \in Q.$ Hence $f$ is strongly $Q$-homogeneous.
\end{proof}

We know by Lemma~\ref{lem: CIMP of a star graph} that the CIM ideal of a star graph (and of a path graph up to $4$ vertices) is the trivial ideal. Thus, one can construct the CIM ideal of $I_T$ of any tree $T$ as an iterated QIG of trivial ideals. By Corollary~\ref{corollary:quasiIndepGrobner}, we have the following corollary. %the Gröbner basis of $I_T$ is quadratic and square-free.
\begin{corollary}
\label{cor:treeUniGB}
Let $T$ be a tree. Then there exists a weight vector $\omega$ such that the reduced Gr\"obner basis of $I_T$ with respect to $\omega$ consists of square-free quadratics. Moreover, these quadratics can be explicitly constructed via iterated quasi-independence gluing. 
\end{corollary}

We close this section with an example involving the path graph. Let $P_n$ denote the path with $n$ vertices and label the vertices by $\{1,\dots,n\}$ in the order of the path. The v-structures on $P_n$ have edges from $i-1$ and $i+1$ to $i$, where the nonleaf vertex is $i \in \{2,\dots,n-1\}$. We cannot have both $i$ and $i+1$ being nonleaf vertices of v-structures in a pattern since the edge containing $i$ and $i+1$ would require two opposing orientations. Consequently we may index variables by sets $S \subseteq \{2,\dots,n-1\}$ of nonconsecutive integers in place of patterns.

%We close this section with an example involving the path graph. Let $P_n$ denote the path with $n$ vertices and label the vertices by $\{1,\dots,n\}$ in the order of the path. The v-structures on $P_n$ can have center vertices in $\{2,\dots,n-1\}$. A v-structure forces edges incident to the center vertex to be direct towards the center. Consequently no two center vertices of v-structures can be labeled by consecutive integers. Rather than index variables by their patterns, we index variables by sets $S \subseteq \{2,\dots,n-1\}$ containing no consecutive integers. 

\begin{example}[6 Vertex Path]
For this example we consider $I_{P_3} \subseteq \kk[x]$, $I_{P_5} \subseteq \kk[y]$, and $I_{P_6} \subseteq \kk[z]$. We glue a path with 3 vertices to $P_5$ to form $P_6$, as in Figure~\ref{fig:6VertexPathGluingRule}. This gluing yields 3 generators from quasi-independence:
\[z_\emptyset z_{25} - z_2z_5, z_\emptyset z_{35} - z_3z_5, z_2z_{35} - z_3 z_{25}.\]

Since $P_3$ is a star graph, $I_{P_3} = 0$, so we only consider lifts of Gr\"obner basis elements in $I_{P_5}$. One can show that the Gr\"obner basis for $I_{P_5}$ consists of just the polynomial $y_\emptyset y_{24} - y_2 y_4$, since it can be seen as the QIG of $I_{P_4}$ with the ideal of the star graph $I_{P_3}$ which are both trivial ideals. Both monomials can lift by the unordered tuples of lower indices $(\emptyset,\emptyset)$ and $(\emptyset,5)$, which result in the following lifts of the binomial:
\[z_\emptyset z_{24} - z_2 z_4, z_5 z_{24} - z_{25} z_4.\]
The 5 binomials we have constructed form a Gr\"obner basis for $I_{P_6}$.
\end{example}

\begin{figure}
\centering
\begin{tikzpicture}[scale = 0.55]

\begin{scope}[shift={(0,2.25)}]
\draw (0.5,-0.866)--(1,0)--(0.5,0.866)--(-0.5,0.866)--(-1,0)--(-0.5,-0.866);

\node at (-0.85,-0.866) {1};
\node at (-1.3,0) {2};
\node at (-0.85,0.866) {3};
\node at (0.85,0.866) {4};
\node at (1.3,0) {5};
\node at (0.85,-0.866) {6};
\end{scope}

\draw [thick] (2,-3.4)--(1.6,-3.4)--(1.6,1.1)--(2,1.1);
\draw [thick] (14,-3.4)--(14.4,-3.4)--(14.4,1.1)--(14,1.1);

%ROW INDICES
\begin{scope}[shift={(0,0)}]
\draw (0.5,-0.866)--(1,0)--(0.5,0.866);

\foreach \r in {1,...,5}{
\draw [shift={(0.5 + 2.5*\r,0)}] (0.5,-0.866)--(1,0)--(0.5,0.866);
}
\end{scope}

\begin{scope}[shift={(0,-2.25)}]
\draw [-Stealth] (0.5,-0.866)--(1,0);
\draw [-Stealth] (0.5,0.866)--(1,0);

\foreach \r in {1,...,3}{
\draw [shift={(0.5 + 2.5*\r,0)},-Stealth] (0.5,-0.866)--(1,0);
\draw [shift={(0.5 + 2.5*\r,0)},-Stealth] (0.5,0.866)--(1,0);
}
\end{scope}

%COLUMN INDICES
\begin{scope}[shift = {(3,0)}]
\foreach \r in {-1,0,1}{
\draw [shift = {(0,2.25*\r)}] (1,0)--(0.5,0.866)--(-0.5,0.866)--(-1,0)--(-0.5,-0.866);
}
\end{scope}

\begin{scope}[shift = {(5.5,0)}]
\foreach \r in {-1,0,1}{
\draw [shift = {(0,2.25*\r)}] (1,0)--(0.5,0.866)--(-0.5,0.866)--(-1,0)--(-0.5,-0.866);
\draw [shift = {(0,2.25*\r)},-Stealth] (-0.5,-0.866)--(-1,0);
\draw [shift = {(0,2.25*\r)},-Stealth] (-0.5,0.866)--(-1,0);
}
\end{scope}

\begin{scope}[shift = {(8,0)}]
\foreach \r in {-1,0,1}{
\draw [shift = {(0,2.25*\r)}] (1,0)--(0.5,0.866)--(-0.5,0.866)--(-1,0)--(-0.5,-0.866);
\draw [shift = {(0,2.25*\r)},-Stealth] (0.5,0.866)--(-0.5,0.866);
\draw [shift = {(0,2.25*\r)},-Stealth] (-1,0)--(-0.5,0.866);
}
\end{scope}

\begin{scope}[shift = {(10.5,0)}]
\foreach \r in {0,1}{
\draw [shift = {(0,2.25*\r)}] (1,0)--(0.5,0.866)--(-0.5,0.866)--(-1,0)--(-0.5,-0.866);
\draw [shift = {(0,2.25*\r)},-Stealth] (-0.5,0.866)--(0.5,0.866);
\draw [shift = {(0,2.25*\r)},-Stealth] (1,0)--(0.5,0.866);
}
\end{scope}

\begin{scope}[shift = {(13,0)}]
\foreach \r in {0,1}{
\draw [shift = {(0,2.25*\r)}] (1,0)--(0.5,0.866)--(-0.5,0.866)--(-1,0)--(-0.5,-0.866);
\draw [shift = {(0,2.25*\r)},-Stealth] (-0.5,0.866)--(0.5,0.866);
\draw [shift = {(0,2.25*\r)},-Stealth] (1,0)--(0.5,0.866);
\draw [shift = {(0,2.25*\r)},-Stealth](-0.5,-0.866)--(-1,0);
\draw [shift = {(0,2.25*\r)},-Stealth](-0.5,0.866)--(-1,0);
}
\end{scope}

%STRUCTURAL ZEROS
\node at (10.5,-2.25) {0};
\node at (13,-2.25) {0};

\end{tikzpicture}
\caption{A QIG yielding the characteristic imset ideal of the path with 6 vertices.}
\label{fig:6VertexPathGluingRule}
\end{figure}
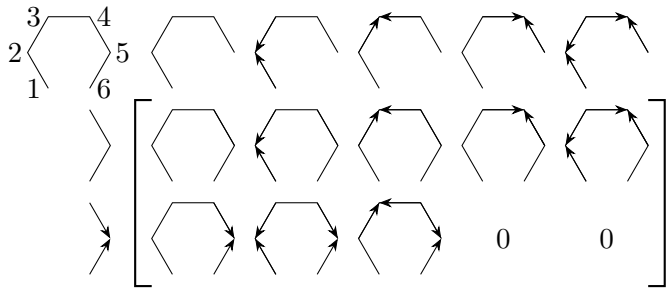

%\section{Characteristic Imset Ideal of the Cycle}
\section{Weak Q-homogeneity, Cycles, and Future Directions}
\label{section:CIMCycle}
In this section we examine the characteristic imset ideal of a cycle. We show that the ideal can be obtained by taking the quasi-independence gluing of two paths but the Gr\"obner basis of the path constructed via QIG is not weakly Q-homogeneous. We end by exploring possible directions for finding a weakly Q-homogeneous Gr\"obner basis for the path. 

We begin with a general gluing of two path graphs, of the form pictured in Figure~\ref{fig:doubleParting}. This is similar to the partings for trees, except we part at two edges. Since the end points of paths are not glued together and since we glue along two edges, each path must have at least four vertices.

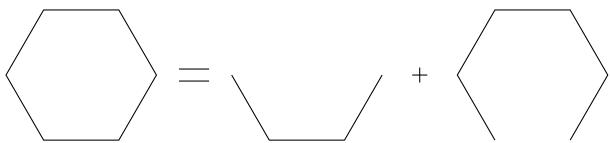
\begin{figure}[h]
\centering
\begin{tikzpicture}
\draw (0.5,-0.866)--(1,0)--(0.5,0.866)--(-0.5,0.866)--(-1,0)--(-0.5,-0.866)--cycle;

\draw (1.3,0.08)--(1.7,0.08);
\draw (1.3,-0.08)--(1.7,-0.08);

\draw [xshift=3cm] (-1,0)--(-0.5,-0.866)--(0.5,-0.866)--(1,0);

\draw (4.5,-0.1)--(4.5,0.1);
\draw (4.4,0)--(4.6,0);

\draw [xshift=6cm] (0.5,-0.866)--(1,0)--(0.5,0.866)--(-0.5,0.866)--(-1,0)--(-0.5,-0.866);
\end{tikzpicture}
\caption{The gluing that we consider in Section~\ref{section:CIMCycle}.}
\label{fig:doubleParting}
\end{figure}

Label $P_n$ as in the previous section, and consider $I_{P_n}$ in variables $y_S$ for $S \subseteq \{2,\dots,n-1\}$ a set of pairwise non-consecutive integers. Let $P_m'$ denote the path on $m$ vertices, which we label $\{n-1,n, n+1, \dots,n+m-4,1,2\}$ in the order of the path. Similar to $P_n$ we index variable $x_S$ by sets $S \subseteq \{n-1,n,\dots,n+m-4,1,2\}$ containing no two integers which are consecutive in the written order. 
We now define a gluing rule $Q_{m,n}$ to be the set of $(\patp',\patp) \in \meq(P_m') \times \meq(P_n)$ such that both of the following hold
\begin{enumerate}
    \item for all v-structures $i \rightarrow k \leftarrow j$ and $i' \rightarrow k' \leftarrow j'$ in $\patp$ or $\patp'$ the numbers $k$ and $k'$ are not cyclically consecutive,
    \item either $\patp'$ or $\patp$ contains a v-structure.
\end{enumerate}

The first condition guarantees that we do not glue edges together with forced opposing orientations while the second condition prevents an additional illegal gluing. Indeed, if we direct the edges of the cycle, the only way that no v-structures can appear is if the edges form a directed cycle, which is not a DAG. Then $I_{C_{n+m-4}}$ is an ideal in the polynomial ring with variables $z_S$ for $S \subseteq \{1,\dots,n+m-4\}$ containing no cyclically consecutive integers modulo $n+m-4$.

\begin{theorem}
\label{thm:cimCycleQIG}
Let $m,n \geq 4$. Then $I_{C_{n+m-4}} = I_{P_m'} \times_{Q_{m,n}} I_{P_n}$.
\end{theorem}

\begin{proof}
The proof is analogous to the proof of Theorem~\ref{thm:cimTreeQIG}. The main difference is that we must square both $t_{12}$ and $t_{n-1,n}$, since we glue along two edges.
\end{proof}

The previous theorem shows that the cycle is a quasi-independence gluing but to compute a Gr\"obner basis with Theorem \ref{thm:gbofquasi-independence}
we also need that there are Gr\"obner bases for both $I_{P_m'}$ and $I_{P_n}$ which are weakly $Q_{m,n}$-homogeneous with respect to some term order. For the remainder of this section we investigate combinatorial and polyhedral restrictions on the weight order that are imposed by weak $Q_{m,n}$-homogeneity. 

\begin{lemma}
\label{lemma:emptySetQHom}
Let $I_{P_n} \subseteq \kk[y_S]$,  $\omega$ be a weight order on $\kk[y]$, and $F$ be a Gr\"obner basis for $I_{P_n}$ with respect to $\omega$. If $I_{P_n}$ is weakly $Q_{m,n}$-homogeneous with respect to $\omega$ then for every binomial in $F$ of the form
$
y_{\emptyset}y_{S_2} \ldots y_{S_k} - y_{T_1}y_{T_2} \ldots y_{T_k}
$
where $T_i \neq \emptyset$, it holds that
\[
y_{T_1}y_{T_2} \ldots y_{T_k}  \leq_w y_{\emptyset}y_{S_2} \ldots y_{S_k}. 
\]
\end{lemma}
\begin{proof}
First, note that the variable $y_{\emptyset}$ is the only variable which does not lift by $\emptyset$. Thus if every $T_i \neq \emptyset$ then the monomial $y_{T_1}y_{T_2} \ldots y_{T_k}$ lifts by indices $(\emptyset, \emptyset, \ldots, \emptyset)$ but the monomial $y_{\emptyset}y_{S_2} \ldots y_{S_k}$ does not. Thus $F$ can only be weakly $Q_{m,n}$-homogeneous if
$y_{\emptyset}y_{S_2} \ldots y_{S_k}$ is the leading term. 
\end{proof}

Lemma~\ref{lemma:emptySetQHom} implies the following restriction on the sizes of the paths that we glue together.

\begin{proposition}
Suppose that $m \geq 5$ and $n \geq 5$. Then there is no term order such that $I_{P_n}$ is weakly $Q_{m,n}$-homogeneous.     
\end{proposition}
\begin{proof}
Since $n \geq 5$, $I_{P_n} \neq 0$ and we have the following generator of $I_{P_n}$:
\[\underline{y_\emptyset y_{2,n-1}} - y_2 y_{n-1}.\]
Observe that every Gr\"obner basis of $I_{P_n}$ must contain this binomial since no other binomial in the ideal contains either of these monomials. By Lemma~\ref{lemma:emptySetQHom} the underlined term must be the leading term.
 
%Regardless of whether $m \geq 5$, we can lift $y_2 y_{n-1}$ by indices $(\emptyset, \emptyset)$ to obtain the monomial $z_2 z_{n-1}$ in ambient ring of $I_{C_{n+m-4}}$. However, $y_\emptyset$ can only be lifted by indices $S \neq \emptyset$ since v-structures must exist in the cycle. Consequently we see that $y_\emptyset y_{2,n-1}$ must be the leading monomial if this binomial is to be weakly $Q_{m,n}$-homogeneous.

But if $m \geq 5$, then a pattern on $P'_m$ could have v-structures at both end points $\{1,n\}$ of $P_n$. We can then lift $y_\emptyset y_{2,n-1}$ by $(\{1,n\},\emptyset)$ to $z_{1,n} z_{2,n-1}$. However, both $2$ and $n-1$ are adjacent to some element of $\{1,n\}$, and so the other monomial does not lift. Since each monomial lifts by an indexing set that the other does not lift by, this polynomial cannot be weakly $Q_{m,n}$-homogeneous with respect to any term order.   
\end{proof}

We consequently restrict our focus to $m=4$. We suppress $m$ from the notation, instead considering the path $P'$ of length $4$ and weak $Q_n$-homogeneity. Note that $\CIM_{P'}$ is a simplex, and so $I_{P'} = 0$. So we need only investigate weak $Q_n$-homogeneity for a Gr\"obner basis of $I_{P_n}$. We now show that any term order producing a weakly $Q_n$-homogeneous Gr\"obner basis cannot be constructed via the iterated QIG in Section \ref{section:CIMTree}.

\begin{lemma}
For $n \geq 6$, there does not exist a Gr\"obner basis of $I_{P_n}$ produced by iterated QIG that is also weakly $Q_n$-homogeneous.
\end{lemma}

\begin{proof}
Observe that every Gr\"obner basis of $I_{P_{n-1}}$ produced by iterative QIG contains the polynomial $y_\emptyset y_{24} - y_2 y_4$. Let $\omega$ be any weight such that $y_2 y_4 \leq_\omega y_\emptyset y_{24}$ since by Lemma \ref{lemma:emptySetQHom} we know $\omega$ must select monomials with $y_\emptyset$ as the leading term. Then QIG produces the following two lifts in $I_{P_n}$:
\[\underline{z_\emptyset z_{24}} - z_2 z_4 \quad \text{and} \quad \underline{z_{n-1} z_{24}} - z_{2,n-1} z_4.\]
Lifting does not change the leading term (regardless of the term order chosen on $I_{P'}$), and so the underlined terms are leading terms. 

Now we consider lifting the polynomial  $z_{n-1} z_{24} - z_{2,n-1} z_4 \in I_{P_n}$ to $I_{C_n}$. The leading monomial $z_{n-1} z_{24}$ lifts by indices $\{(\emptyset, \emptyset), (\emptyset,n), (1,\emptyset),
(1,n)\}$ while the the trailing monomial $z_{2,n-1} z_4$ lifts by indices $\{(\emptyset, \emptyset), (\emptyset,n), (1,\emptyset)\}$. Consequently $z_{n-1}z_{24} - z_{2,n-1}z_4$ is not weakly $Q_n$-homogeneous.
\end{proof}

Figure~\ref{fig:hexagonGluingRule} illustrates Lemma~\ref{lemma:emptySetQHom}. The monomial map of the characteristic imset ideal factors according to $Q_n$, but weak $Q_n$-homogeneity does not hold for the Gr\"obner basis of the path that we constructed in Section~\ref{section:CIMTree}.

\begin{figure}
\centering
\begin{tikzpicture}[scale = 0.55]

\begin{scope}[shift={(0,2.25)}]
\draw (0.5,-0.866)--(1,0)--(0.5,0.866)--(-0.5,0.866)--(-1,0)--(-0.5,-0.866)--cycle;

\node at (-0.85,-0.866) {1};
\node at (-1.3,0) {2};
\node at (-0.85,0.866) {3};
\node at (0.85,0.866) {4};
\node at (1.3,0) {5};
\node at (0.85,-0.866) {6};
\end{scope}

\draw [thick] (2,-5.6)--(1.6,-5.6)--(1.6,1.1)--(2,1.1);
\draw [thick] (21.5,-5.6)--(21.9,-5.6)--(21.9,1.1)--(21.5,1.1);

%ROW INDICES
\foreach \r in {0,...,2}{
\begin{scope}[shift = {(0,-2.25*\r)}]
\draw (-1,0)--(-0.5,-0.866)--(0.5,-0.866)--(1,0);
\end{scope}
}

\begin{scope}[shift = {(0,-2.25)}]
\draw [-Stealth] (-0.5,-0.866)--(0.5,-0.866);
\draw [Stealth-] (0.5,-0.866)--(1,0);
\end{scope}

\begin{scope}[shift = {(0,-4.5)}]
\draw [-Stealth] (0.5,-0.866)--(-0.5,-0.866);
\draw [Stealth-] (-0.5,-0.866)--(-1,0);
\end{scope}

\begin{scope}[shift = {(0.5,0)}]

%COLUMN INDICES
\foreach \r in {1,...,8}{
\begin{scope}[shift = {(2.5*\r,2.25)}]
\draw (0.5,-0.866)--(1,0)--(0.5,0.866)--(-0.5,0.866)--(-1,0)--(-0.5,-0.866);
\end{scope}
}

\begin{scope}[shift = {(5,2.25)}]
\draw [Stealth-] (0.5,0.866)--(1,0);
\draw [Stealth-] (0.5,0.866)--(-0.5,0.866);
\end{scope}

\begin{scope}[shift = {(7.5,2.25)}]
\draw [Stealth-] (-0.5,0.866)--(-1,0);
\draw [Stealth-] (-0.5,0.866)--(0.5,0.866);
\end{scope}

\begin{scope}[shift = {(10,2.25)}]
\draw [-Stealth] (0.5,-0.866)--(1,0);
\draw [-Stealth] (0.5,0.866)--(1,0);
\end{scope}

\begin{scope}[shift = {(12.5,2.25)}]
\draw [Stealth-] (-0.5,0.866)--(-1,0);
\draw [Stealth-] (-0.5,0.866)--(0.5,0.866);
\draw [-Stealth] (0.5,-0.866)--(1,0);
\draw [-Stealth] (0.5,0.866)--(1,0);
\end{scope}

\begin{scope}[shift = {(15,2.25)}]
\draw [-Stealth] (-0.5,0.866)--(-1,0);
\draw [-Stealth] (-0.5,0.866)--(0.5,0.866);
\draw [-Stealth] (-0.5,-0.866)--(-1,0);
\draw [Stealth-] (0.5,0.866)--(1,0);
\end{scope}

\begin{scope}[shift = {(17.5,2.25)}]
\draw [-Stealth] (-0.5,-0.866)--(-1,0);
\draw [-Stealth] (-0.5,0.866)--(-1,0);
\end{scope}

\begin{scope}[shift = {(20,2.25)}]
\draw [-Stealth] (-0.5,0.866)--(-1,0);

\draw [-Stealth] (-0.5,-0.866)--(-1,0);
\draw [-Stealth] (0.5,0.866)--(1,0);
\draw [-Stealth] (0.5,-0.866)--(1,0);
\end{scope}

%ROW 1
\foreach \r in {2,...,8}{
\begin{scope}[shift = {(2.5*\r,0)}]
\draw (0.5,-0.866)--(1,0)--(0.5,0.866)--(-0.5,0.866)--(-1,0)--(-0.5,-0.866)--cycle;
\end{scope}
}

\begin{scope}[shift = {(2.5,0)}]
\node at (0,0) {0};
\end{scope}

\begin{scope}[shift = {(5,0)}]
\draw [Stealth-] (0.5,0.866)--(1,0);
\draw [Stealth-] (0.5,0.866)--(-0.5,0.866);
\end{scope}

\begin{scope}[shift = {(7.5,0)}]
\draw [Stealth-] (-0.5,0.866)--(-1,0);
\draw [Stealth-] (-0.5,0.866)--(0.5,0.866);
\end{scope}

\begin{scope}[shift = {(10,0)}]
\draw [-Stealth] (0.5,-0.866)--(1,0);
\draw [-Stealth] (0.5,0.866)--(1,0);
\end{scope}

\begin{scope}[shift = {(12.5,0)}]
\draw [Stealth-] (-0.5,0.866)--(-1,0);
\draw [Stealth-] (-0.5,0.866)--(0.5,0.866);
\draw [-Stealth] (0.5,-0.866)--(1,0);
\draw [-Stealth] (0.5,0.866)--(1,0);
\end{scope}

\begin{scope}[shift = {(15,0)}]
\draw [-Stealth] (-0.5,0.866)--(-1,0);
\draw [-Stealth] (-0.5,0.866)--(0.5,0.866);
\draw [-Stealth] (-0.5,-0.866)--(-1,0);
\draw [Stealth-] (0.5,0.866)--(1,0);
\end{scope}

\begin{scope}[shift = {(17.5,0)}]
\draw [-Stealth] (-0.5,-0.866)--(-1,0);
\draw [-Stealth] (-0.5,0.866)--(-1,0);
\end{scope}

\begin{scope}[shift = {(20,0)}]
\draw [-Stealth] (-0.5,0.866)--(-1,0);

\draw [-Stealth] (-0.5,-0.866)--(-1,0);
\draw [-Stealth] (0.5,0.866)--(1,0);
\draw [-Stealth] (0.5,-0.866)--(1,0);
\end{scope}

%ROW 2
\foreach \r in {1,2,3,6,7}{
\begin{scope}[shift = {(2.5*\r,-2.25)}]
\draw (0.5,-0.866)--(1,0)--(0.5,0.866)--(-0.5,0.866)--(-1,0)--(-0.5,-0.866)--cycle;
\end{scope}
}

\begin{scope}[shift = {(2.5,-2.25)}]
\draw [-Stealth] (-0.5,-0.866)--(0.5,-0.866);
\draw [-Stealth] (1,0)--(0.5,-0.866);
\end{scope}

\begin{scope}[shift = {(5,-2.25)}]
\draw [Stealth-] (0.5,0.866)--(1,0);
\draw [Stealth-] (0.5,0.866)--(-0.5,0.866);
\draw [-Stealth] (-0.5,-0.866)--(0.5,-0.866);
\draw [-Stealth] (1,0)--(0.5,-0.866);
\end{scope}

\begin{scope}[shift = {(7.5,-2.25)}]
\draw [Stealth-] (-0.5,0.866)--(-1,0);
\draw [Stealth-] (-0.5,0.866)--(0.5,0.866);
\draw [-Stealth] (-0.5,-0.866)--(0.5,-0.866);
\draw [-Stealth] (1,0)--(0.5,-0.866);
\end{scope}

\begin{scope}[shift = {(10,-2.25)}]
\node at (0,0) {0};
\end{scope}

\begin{scope}[shift = {(12.5,-2.25)}]
\node at (0,0) {0};
\end{scope}

\begin{scope}[shift = {(15,-2.25)}]
\draw [-Stealth] (-0.5,0.866)--(-1,0);
\draw [-Stealth] (-0.5,0.866)--(0.5,0.866);
\draw [-Stealth] (-0.5,-0.866)--(-1,0);
\draw [Stealth-] (0.5,0.866)--(1,0);
\draw [-Stealth] (-0.5,-0.866)--(0.5,-0.866);
\draw [-Stealth] (1,0)--(0.5,-0.866);
\end{scope}

\begin{scope}[shift = {(17.5,-2.25)}]
\draw [-Stealth] (-0.5,-0.866)--(-1,0);
\draw [-Stealth] (-0.5,0.866)--(-1,0);
\draw [-Stealth] (-0.5,-0.866)--(0.5,-0.866);
\draw [-Stealth] (1,0)--(0.5,-0.866);
\end{scope}

\begin{scope}[shift = {(20,-2.25)}]
\node at (0,0) {0};
\end{scope}

%ROW 3
\foreach \r in {1,2,3,4,5}{
\begin{scope}[shift = {(2.5*\r,-4.5)}]
\draw (0.5,-0.866)--(1,0)--(0.5,0.866)--(-0.5,0.866)--(-1,0)--(-0.5,-0.866)--cycle;
\end{scope}
}

\begin{scope}[shift = {(2.5,-4.5)}]
\draw [-Stealth] (0.5,-0.866)--(-0.5,-0.866);
\draw [-Stealth] (-1,0)--(-0.5,-0.866);
\end{scope}

\begin{scope}[shift = {(5,-4.5)}]
\draw [Stealth-] (0.5,0.866)--(1,0);
\draw [Stealth-] (0.5,0.866)--(-0.5,0.866);
\draw [-Stealth] (0.5,-0.866)--(-0.5,-0.866);
\draw [-Stealth] (-1,0)--(-0.5,-0.866);
\end{scope}

\begin{scope}[shift = {(7.5,-4.5)}]
\draw [Stealth-] (-0.5,0.866)--(-1,0);
\draw [Stealth-] (-0.5,0.866)--(0.5,0.866);
\draw [-Stealth] (0.5,-0.866)--(-0.5,-0.866);
\draw [-Stealth] (-1,0)--(-0.5,-0.866);
\end{scope}

\begin{scope}[shift = {(10,-4.5)}]
\draw [-Stealth] (0.5,-0.866)--(1,0);
\draw [-Stealth] (0.5,0.866)--(1,0);
\draw [-Stealth] (0.5,-0.866)--(-0.5,-0.866);
\draw [-Stealth] (-1,0)--(-0.5,-0.866);
\end{scope}

\begin{scope}[shift = {(12.5,-4.5)}]
\draw [-Stealth] (0.5,0.866)--(1,0);
\draw [-Stealth] (0.5,0.866)--(-0.5,0.866);
\draw [-Stealth] (0.5,-0.866)--(1,0);
\draw [Stealth-] (-0.5,0.866)--(-1,0);
\draw [-Stealth] (0.5,-0.866)--(-0.5,-0.866);
\draw [-Stealth] (-1,0)--(-0.5,-0.866);
\end{scope}

\begin{scope}[shift = {(15,-4.5)}]
\node at (0,0) {0};
\end{scope}

\begin{scope}[shift = {(17.5,-4.5)}]
\node at (0,0) {0};
\end{scope}

\begin{scope}[shift = {(20,-4.5)}]
\node at (0,0) {0};
\end{scope}
\end{scope}

\end{tikzpicture}
\caption{A QIG that yields the characteristic imset ideal of the cycle with 6 vertices. One can visually check that there is no term order constructed by QIG such that $z_\emptyset z_{24} - y_2 y_4$ and $z_5z_{24} - z_{25}z_4$ are not both weakly $Q_6$-homogeneous.}
    \label{fig:hexagonGluingRule}
\end{figure}
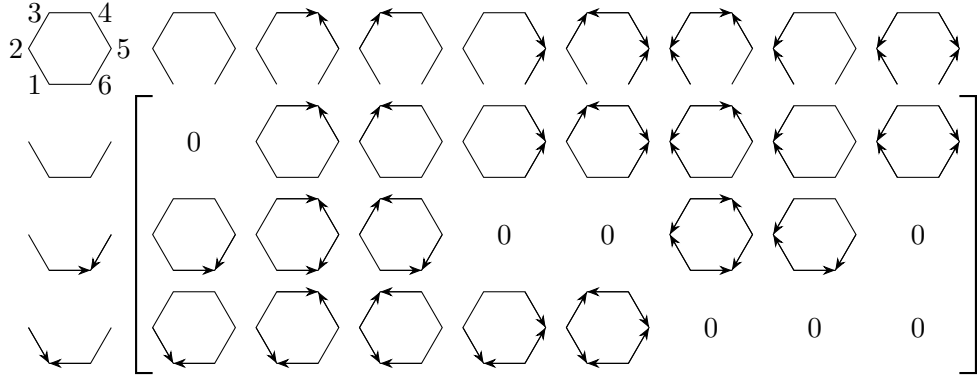

The previous lemma shows that it is impossible for iterative quasi-independence gluing to produce a Gr\"obner basis for $I_{P_n}$ which can then be used to compute $I_{C_n}$. 
However, our computations for $n=6,7$ suggest that there do exist Gr\"obner bases for $I_{P_n}$ which are weakly $Q_n$-homogeneous. 

\begin{problem}
Find a term order $\omega$ on $I_{P_n}$ such that a Gr\"obner basis $F$ of $I_{P_n}$ with respect to $\omega$ is weakly $Q_n$-homogeneous. 
\end{problem}

A solution to this problem would allow one to construct the Gr\"obner basis for the cycle ideal $I_{C_n}$ using Theorem \ref{thm:quasiIndepPseudoGrobner}. One potential first step is to begin with a Gr\"obner basis $G$ produced by QIG and then construct a new term order $\omega$, which does not come from QIG, such that $G$ is a weakly $Q_n$-homogeneous generating set with respect to $\omega$. Since weak $Q_n$-homogeneity corresponds to homogeneous linear inequalities on $\omega$, finding a term order $\omega$ just amounts to finding an interior point in the corresponding cone. The existence of such a term order would immediately allow one to find a generating set for $I_{C_n}$ using QIG. To find a Gr\"obner basis one would instead need to compute a Gr\"obner basis for $I_{P_n}$ with respect to the term order $\omega$ and certify that any new polynomials which arise in the generating set are also weakly $Q_n$-homogeneous with respect to $\omega$. 

Another direction for further work is finding other classes of ideals which arise via quasi-independence gluing. It would be particularly interesting to find ideals which arise via weakly $Q$-homogeneous quasi-independence gluing. This is summarized in the following question. 

\begin{question}
Are there other interesting families of ideals which can be constructed via iterated quasi-independence gluing?
\end{question}

\section*{Acknowledgements}
We are grateful to the organizers of
Algebraic Statistics Workshop 2022 at University of Hawai'i at Manoa, Honolulu, for fostering an excellent research atmosphere that enabled the authors of this paper to collaborate.
We would like to thank Seth Sullivant for his comments on an earlier version of this manuscript. Joseph Johnson was partially supported by the North Carolina State University Graduate School Summer Fellowship. Liam Solus was partially supported by the Wallenberg Autonomous Systems and Software Program (WASP) funded by the Knut and Alice
Wallenberg Foundation, the G\"oran Gustafsson Prize for Young Researchers, and Starting Grant No.\ 2019-05195 from The Swedish Research Council (Vetenskapsr\aa{}det).

\bibliography{refs.bib}{}

\begin{thebibliography}{10}

\bibitem{ananiadi2021grobner}
Lamprini Ananiadi and Eliana Duarte.
\newblock Gr\"{o}bner bases for staged trees.
\newblock {\em Algebr. Stat.}, 12(1):1--20, 2021.

\bibitem{andersson1997markov}
Steen~A. Andersson, David Madigan, and Michael~D. Perlman.
\newblock A characterization of {M}arkov equivalence classes for acyclic
  digraphs.
\newblock {\em Annals of Statistics}, 25(2):505--541, 1997.

\bibitem{aoki2012markov}
Satoshi Aoki, Hisayuki Hara, and Akimichi Takemura.
\newblock {\em Markov bases in algebraic statistics}.
\newblock Springer Series in Statistics. Springer, New York, 2012.

\bibitem{chickering2002optimal}
David~Maxwell Chickering.
\newblock Optimal structure identification with greedy search.
\newblock {\em Journal of machine learning research}, 3(Nov):507--554, 2002.

\bibitem{coons2021quasi}
Jane~Ivy Coons and Seth Sullivant.
\newblock Quasi-independence models with rational maximum likelihood estimator.
\newblock {\em J. Symbolic Comput.}, 104:917--941, 2021.

\bibitem{coons2021toric}
Jane~Ivy Coons and Seth Sullivant.
\newblock Toric geometry of the {C}avender-{F}arris-{N}eyman model with a
  molecular clock.
\newblock {\em Adv. in Appl. Math.}, 123:Paper No. 102119, 54, 2021.

\bibitem{cummings2021invariants}
Joseph Cummings, Benjamin Hollering, and Christopher Manon.
\newblock Invariants for level-1 phylogenetic networks under the
  cavendar-farris-neyman model, 2021.

\bibitem{de1995grobner}
Jes{\'u}s~A De~Loera, Bernd Sturmfels, and Rekha~R Thomas.
\newblock Gr{\"o}bner bases and triangulations of the second hypersimplex.
\newblock {\em Combinatorica}, 15(3):409--424, 1995.

\bibitem{engstrom2014multigraded}
Alexander Engstr\"{o}m, Thomas Kahle, and Seth Sullivant.
\newblock Multigraded commutative algebra of graph decompositions.
\newblock {\em J. Algebraic Combin.}, 39(2):335--372, 2014.

\bibitem{GITLER2010430}
Isidoro Gitler, Enrique Reyes, and Rafael~H. Villarreal.
\newblock Ring graphs and complete intersection toric ideals.
\newblock {\em Discrete Mathematics}, 310(3):430--441, 2010.
\newblock Sixth Czech-Slovak International Symposium on Combinatorics, Graph
  Theory, Algorithms and Applications.

\bibitem{Herzog2018}
J{\"u}rgen Herzog, Takayuki Hibi, and Hidefumi Ohsugi.
\newblock {\em Edge Polytopes and Edge Rings}, pages 117--140.
\newblock Springer International Publishing, Cham, 2018.

\bibitem{koller2009probabilistic}
Daphne Koller and Nir Friedman.
\newblock {\em Probabilistic graphical models: principles and techniques}.
\newblock MIT press, 2009.

\bibitem{kuipers2022efficient}
Jack Kuipers, Polina Suter, and Giusi Moffa.
\newblock Efficient sampling and structure learning of bayesian networks.
\newblock {\em Journal of Computational and Graphical Statistics}, pages 1--12,
  2022.

\bibitem{LRS2022a}
Svante Linusson, Petter Restadh, and Liam Solus.
\newblock Greedy causal discovery is geometric.
\newblock {\em To appear in SIAM Journal on Discrete Mathematics}, 2022.

\bibitem{LRS2022b}
Svante Linusson, Petter Restadh, and Liam Solus.
\newblock On the edges of characteristic imset polytopes.
\newblock {\em arXiv preprint arXiv:2209.07579}, 2022.

\bibitem{OHSUGI1999509}
Hidefumi Ohsugi and Takayuki Hibi.
\newblock Toric ideals generated by quadratic binomials.
\newblock {\em Journal of Algebra}, 218(2):509--527, 1999.

\bibitem{pearl2009causality}
Judea Pearl.
\newblock {\em Causality}.
\newblock Cambridge university press, 2009.

\bibitem{por21}
Irem Portakal.
\newblock On rigidity of toric varieties arising from bipartite graphs.
\newblock {\em Journal of Algebra}, 569:784--822, 2021.

\bibitem{rauh2016lifting}
Johannes Rauh and Seth Sullivant.
\newblock Lifting {M}arkov bases and higher codimension toric fiber products.
\newblock {\em J. Symbolic Comput.}, 74:276--307, 2016.

\bibitem{solus2017consistency}
Liam Solus, Yuhao Wang, and Caroline Uhler.
\newblock Consistency guarantees for greedy permutation-based causal inference
  algorithms.
\newblock {\em Biometrika}, 108(4):795--814, 2021.

\bibitem{spirtes1991algorithm}
Peter Spirtes and Clark Glymour.
\newblock An algorithm for fast recovery of sparse causal graphs.
\newblock {\em Social science computer review}, 9(1):62--72, 1991.

\bibitem{studeny2017towards}
Milan Studen{\`y} and James Cussens.
\newblock Towards using the chordal graph polytope in learning decomposable
  models.
\newblock {\em International Journal of Approximate Reasoning}, 88:259--281,
  2017.

\bibitem{studeny2021dual}
Milan Studen{\`y}, James Cussens, and V{\'a}clav Kratochv{\'\i}l.
\newblock The dual polyhedron to the chordal graph polytope and the rebuttal of
  the chordal graph conjecture.
\newblock {\em International Journal of Approximate Reasoning}, 138:188--203,
  2021.

\bibitem{studeny2010characteristic}
Milan Studen{\`y}, Raymond Hemmecke, and Silvia Lindner.
\newblock Characteristic imset: a simple algebraic representative of a bayesian
  network structure.
\newblock In {\em Proceedings of the 5th European workshop on probabilistic
  graphical models}, pages 257--264. HIIT Publications, 2010.

\bibitem{sturmfels1996Grobner}
Bernd Sturmfels.
\newblock {\em Gr\"{o}bner bases and convex polytopes}, volume~8 of {\em
  University Lecture Series}.
\newblock American Mathematical Society, Providence, RI, 1996.

\bibitem{sturmfels2005toric}
Bernd Sturmfels and Seth Sullivant.
\newblock Toric ideals of phylogenetic invariants.
\newblock {\em Journal of computational biology : a journal of computational
  molecular cell biology}, 12:457--81, 06 2005.

\bibitem{sturmfels2008toric}
Bernd Sturmfels and Seth Sullivant.
\newblock Toric geometry of cuts and splits.
\newblock volume~57, pages 689--709. 2008.
\newblock Special volume in honor of Melvin Hochster.

\bibitem{sullivant2007toric}
Seth Sullivant.
\newblock Toric fiber products.
\newblock {\em J. Algebra}, 316(2):560--577, 2007.

\bibitem{tsamardinos2006max}
Ioannis Tsamardinos, Laura~E Brown, and Constantin~F Aliferis.
\newblock The max-min hill-climbing bayesian network structure learning
  algorithm.
\newblock {\em Machine learning}, 65(1):31--78, 2006.

\bibitem{verma1990equivalence}
Thomas Verma and Judea Pearl.
\newblock Equivalence and synthesis of causal models.
\newblock In {\em Proceedings of the Sixth Annual Conference on Uncertainty in
  Artificial Intelligence}, pages 255--270, 1990.

\bibitem{verma1992algorithm}
Thomas Verma and Judea Pearl.
\newblock An algorithm for deciding if a set of observed independencies has a
  causal explanation.
\newblock In {\em Proceedings of the Eighth Annual Conference on Uncertainty in
  Artificial Intelligence}, pages 323--330, 1992.

\bibitem{verma2022equivalence}
Thomas~S Verma and Judea Pearl.
\newblock Equivalence and synthesis of causal models.
\newblock In {\em Probabilistic and Causal Inference: The Works of Judea
  Pearl}, pages 221--236. 2022.

\bibitem{villarreal}
Rafael Villarreal~Rodriguez.
\newblock Rees algebras of edge ideals.
\newblock {\em Communications in Algebra}, 23:3513--3524, 01 1995.

\bibitem{xi2015characteristic}
Jing Xi and Ruriko Yoshida.
\newblock The characteristic imset polytope of bayesian networks with ordered
  nodes.
\newblock {\em SIAM Journal on Discrete Mathematics}, 29(2):697--715, 2015.

\end{thebibliography}
\bibliographystyle{plain}
\end{document}